\def\CC         {{\mathbb C}}
\def\QQ         {{\mathbb Q}}
\def\PP         {{\mathbb P}}

\def\dual	{{\rm \vee}}

\def\dim        {{\rm dim}}

\documentclass{amsart}
\usepackage{amssymb}
\usepackage{verbatim}
\usepackage{eucal}
\usepackage{mathrsfs}
\usepackage{graphicx}
\usepackage{psfrag}
\usepackage{rotating}

\newtheorem{theorem}{Theorem}[section]
\newtheorem{lemma}[theorem]{Lemma}
\newtheorem{proposition}[theorem]{Proposition}
\newtheorem{corollary}[theorem]{Corollary}

\theoremstyle{definition}
\newtheorem{definition}[theorem]{Definition}

\theoremstyle{remark}
\newtheorem{remark}[theorem]{Remark}

\begin{document}

\title{Stringy $E$-functions of Pfaffian-Grassmannian double mirrors}

\author{Lev Borisov}
\address{Department of Mathematics\\
Rutgers University\\
Piscataway, NJ 08854}
\email{borisov@math.rutgers.edu}

\author{Anatoly Libgober}
\address{Department of Mathematics\\
University of Illinois\\
Chicago, IL 60607}
\email{libgober@math.uic.edu}

\thanks{The first author was partially supported by NSF grant DMS-1201466.
The second author was partially supported by a grant from Simons Foundation.}

\begin{abstract}
We establish the equality of stringy $E$-functions for double mirror Calabi-Yau complete intersections in 
the varieties of skew forms of rank at most $2k$ and at most $n-1-2k$ on a vector space of odd dimension $n$.
\end{abstract}

\maketitle

\section{Introduction}

Mirror symmetry in its classical formulation is the statement that certain quantum field theories defined using different 
Calabi-Yau manifolds differ by a switch between the so-called IIA and IIB twists. This physical (or more precisely string theoretical) phenomenon
implies a vast array of consequences for various invariants of the Calabi-Yau manifolds in question. 

\medskip
In recent years, there has been considerable interest in the so-called double mirror phenomenon, which occurs when two different families of
Calabi-Yau varieties $\{X_\alpha\}$ and $\{Y_\alpha\}$ share the same mirror family. In the majority of known 
cases these Calabi-Yau varieties are simply birational to each other. There are, however, a few instances of non-birational Calabi-Yau double mirror manifolds, of which the oldest and most prominent one is the example of R\o dland, called Pfaffian-Grassmannian correspondence. In this paper we will explore 
the generalization of this example to higher dimensions suggested by Kuznetsov, see  \cite{Kuznetsov2}. We will prove that the stringy $E$-functions
of the expected double mirror varieties coincide.

\medskip
We will now describe the original example of R\o dland.
Let $V$ be a complex vector space of dimension $n=7$ and $W$ be a generic subspace of dimension $7$ of the space $ \Lambda^2 V^\dual$ of skew forms on $V$. To these data one associates a complete intersection Calabi-Yau threefold $X_W\subset
G(2,V)$ and another Calabi-Yau threefold $Y_W\subset \PP W$ which is the locus of degenerate forms. For a generic choice of $W$, these $X_W$ and $Y_W$ are smooth Calabi-Yau threefolds with Hodge numbers $(h^{1,1},h^{1,2}) = (1,50)$. Their double mirror status was first suggested by \cite{Rodland} and then further solidified by \cite{BC,HoriTong,Kuznetsov06}. 

\medskip
Analogous construction works for an arbitrary odd $n\geq 5$. We get two families of  Calabi-Yau varieties $\{X_W\}$ and $\{Y_W\}$
of dimension $(n-4)$ 
and we can try to verify various mathematical consequences of their conjectural double mirror status. The most accessible such 
property is equality of their Hodge numbers. However, for $n\geq 11$, the Pfaffian side $Y_W$ is singular, so its Hodge numbers need to be generalized to stringy Hodge numbers defined in \cite{Batyrev.1}.
The first important result of our paper is the following:

\bigskip
{\bf Theorem \ref{main}.}
For any odd $n\geq 5$ we have the equality of Hodge numbers
$$
h^{p,q}(X_W) = h^{p,q}_{st}(Y_W).
$$

\bigskip
While the result  of Theorem \ref{main} is not particularly surprising, it requires an elaborate calculation which 
involves the log resolution of the Pfaffian variety given in terms of the so-called spaces of complete skew forms \cite{Bertram, Thaddeus}. In the process we end up calculating stringy Hodge numbers of Pfaffian varieties by an inductive argument.

\medskip
There is a way to further generalize the Pfaffian-Grassmannian correspondence which we will now describe.
The Pfaffian-Grassmannian correspondence can be viewed as a particular case of a more general correspondence between Calabi-Yau complete intersections 
$X_W$ and $Y_W$ in dual Pfaffian varieties $Pf(2k,V^\dual)$ and $Pf(n-1-2k,V)$ for a vector space $V$ of odd dimension $n$. Here $Pf(2k,V^\dual)$ is the $k$-th secant variety
of $G(2,V)\subseteq \PP\Lambda^2 V$. We define these varieties $X_W$ and $Y_W$ in Section \ref{sec7} and  eventually prove 
the following, rather more technical result.

\bigskip
{\bf Theorem \ref{main.main}.}
The varieties $X_W$ and $Y_W$ have well-defined stringy Hodge numbers. Moreover, there holds 
$$
h_{st}^{p,q}(X_W)=h_{st}^{p.q}(Y_W).
$$

\bigskip
We chose to discuss the easier case of the Pfaffian-Grassmannian correspondence in more detail, so that the reader can focus on
it first and only then continue to the general case. In Section \ref{sec2} we define the varieties $X_W$ and $Y_W$, prove their basic properties,
and formulate the main result Theorem \ref{main}. We also recall the definition of stringy $E$-functions and discuss the case of 
Zariski locally trivial resolutions. In Section \ref{sec.log} we calculate a log resolution of the Pfaffian variety in terms of the spaces of complete skew forms. It includes a delicate calculation of the discrepancies of the exceptional divisors. Section \ref{sec4} contains an inductive calculation of the stringy  $E$-functions of Pfaffians in odd dimensional spaces. We get a remarkably simple formula 
for it in Theorem \ref{stpf}. Section \ref{proof} finishes the argument by considering projections of the Cayley hypersurface of $X_W$. 
Section \ref{sec.even} describes the analogous construction in the case of even $n$. For even $n$, the varieties $X_W$ and $Y_W$ have different dimensions and can thus be only double mirrors in some generalized sense. Moreover, it appears that the definition of 
stringy $E$-function needs to be adjusted for such generalized double mirrors, since the usual stringy $E$-function does not work.

\medskip
We then proceed with the definitions and arguments for the general case.
In Section \ref{sec7} we define the varieties $X_W$ and $Y_W$, prove their basic properties and formulate the second main result Theorem \ref{main.main}.
Section \ref{sec8} proceeds to prove  Theorem \ref{main.main} modulo some technical results relegated 
to Section \ref{Appendix}. Finally, in Section \ref{sec.last} we make a few concluding remarks with the focus on open questions related to our construction.

\medskip
{\bf Acknowledgements.} L.B. would like to thank Emanuel Diaconescu for stimulating conversations and interest in the project.
We thank Alexander Kuznetsov for multiple helpful comments on the first version of the paper.
We also thank Doron Zeilberger for encouragement and Hjalmar Rosengren for generous help with $q$-hypergeometric identities
of Section \ref{Appendix}. Our interactions with Hjalmar Rosengren were greatly facilitated by the website mathoverflow.net. We used 
the software packages MAPLE and PARI/GP to formulate conjectures, although our proofs do not rely on a computer calculation.

\section{Pfaffian and Grassmannian double mirror Calabi-Yau varieties}\label{sec2}
Let $V$ be an $n$-dimensional complex vector space for an odd $n\geq 5$. 
Let $W\subset \Lambda^2V^\dual$ be a generic $n$-dimensional space of  skew forms on $V$. 
To these data we associate two Calabi-Yau varieties
$X_W$ and $Y_W$ as follows.

\begin{itemize}
\item $X_W$ is a subvariety of the Grassmannian $G(2,V)$ of dimension two subspaces $T_2\subset V$.
It is defined as the locus of  $T_2\in G(2,V)$ with $w\Big\vert_{T_2} = 0$ for all $w\in W$.
\item $Y_W$ is a subvariety of the Pfaffian variety $Pf(V)\subset \PP\Lambda^2 V$ of skew forms on $V$
whose rank is less than $n-1$. It is defined as the intersection of $Pf(V)$ with $\PP W\subset \PP\Lambda^2 V$.
\end{itemize}

\begin{proposition}
For a generic choice of $W$ the variety $X_W$ is a smooth variety of dimension $n-4$ with trivial
canonical class, and the variety $Y_W$ is a variety of dimension $n-4$ with at worst Gorenstein singularities 
and trivial canonical class.
\end{proposition}

\begin{proof}
To prove the first statement, observe that $X_W$ is the intersection of $n$ generic hyperplanes 
in $\PP\Lambda^2 V$ with the Grassmannian $G(2,V)$ in its Pl\"ucker embedding. This intersection is smooth
by the Bertini theorem and has trivial canonical class by the adjunction formula and the formula for the 
canonical class of $G(2,V)$.

\smallskip
To prove the second statement, we recall the results of  \cite{BE}. The variety $Pf(V)$ is of codimension three
in $\PP \Lambda^2 V^\dual$. The resolution of the pushforward of the  structure sheaf
 $i_*\mathcal O_{Pf(V)}$ in $\PP \Lambda^2 V^\dual$ is  given by the powers of the universal skew form
 as
 \begin{equation}\label{Beres}
0\to {\mathcal O}(-n)
\longrightarrow
{ \mathcal O}(-\frac {n+1}2)^{\oplus n}
\longrightarrow
{\mathcal O}(-\frac {n-1}2)^{\oplus n}
\longrightarrow
{\mathcal O} 
\longrightarrow
i_*{\mathcal O}_{Pf(V)}\to 0.
\end{equation}
Thus, the variety is Gorenstein and by \cite{Hartshorne} we have 
$$i_*K_{Pf(V)}={\mathcal Ext}^3(i_*{\mathcal O}_{Pf(V)},\mathcal O(-\frac 12n(n-1))).$$
The dual of the first map in \eqref{Beres} is up to a twist the third map, so we see
$i_*K_{Pf(V)} = i_*{\mathcal O}_{Pf(V)}(-\frac 12 n(n-3))$
and 
\begin{equation}\label{canPf}
K_{Pf(V)} = i^*{\mathcal O}(-\frac 12 n(n-3)).
\end{equation}
Since $\PP W$ is a generic subspace of codimension $\frac 12 n(n-3)$ in $\PP\Lambda^2V^\dual$, the
Bertini theorem and adjunction for Gorenstein varieties finishes the proof.
\end{proof}

\begin{remark}
Dimension counts show that $Y_W$ is smooth for $n\leq 9$ and is singular thereafter.
\end{remark}

\begin{remark} 
It is easy to show that $X_W$ and $Y_W$ are Calabi-Yau varieties in the strict sense, namely,
that 
$$
H^i(X_W,{\mathcal O})=H^i(Y_W,{\mathcal O})=0
$$
for all $i=1,\ldots, n-5$.
This follows from the Lefschetz hyperplane theorem in $X_W$ case and the exact sequence \eqref{Beres} in the $Y_W$ case.
\end{remark}

\bigskip
Our interest in the varieties $X_W$ and $Y_W$ stems from the observation of R\o dland \cite{Rodland}
that for $n=7$ the corresponding families of  Calabi-Yau threefolds have the same mirror family. 
From the physicists' point of view, this corresponds to the statement that the string theories with 
targets $X_W$ and $Y_W$ can be obtained from one another by analytic continuation in  the K\"ahler
parameter space. We will not pretend to have a full understanding of the physical meaning of this claim
but will instead refer interested readers to \cite{HoriTong} for more details. We refer to pairs of such varieties as
\emph{double mirror} to each other, as indicated in the title of this section.

\medskip
Double mirror Calabi-Yau manifolds are expected to be intimately related to each other. In particular, 
their Hodge numbers are expected to be the same. In addition, one expects that the bounded 
derived categories of coherent sheaves on $X_W$ and $Y_W$ are equivalent. This, indeed, 
has been verified independently in \cite{BC} and \cite{Kuznetsov06}  in $n=7$ case, thus providing a 
rare example of non-birational derived equivalent manifolds.

\medskip
It is reasonable to conjecture that $X_W$ and $Y_W$ are double mirror to each other for arbitrary $n$.
The mathematical consequences of this statement undoubtedly need to be adjusted due to the presence 
of singularities in $Y_W$. The string theory corrections due to the singularities are not fully understood, however,
there is a robust definition of \emph{stringy} Hodge numbers of singular varieties, due to Batyrev \cite{Batyrev.1}.
With this in mind, it becomes natural to conjecture and then prove the following result, which is the main
focus of this paper.  The proof of the theorem is postponed until Section \ref{proof}.

\begin{theorem}\label{main}
For any odd $n\geq 5$ we have the equality of Hodge numbers
$$
h^{p,q}(X_W) = h^{p,q}_{st}(Y_W).
$$
\end{theorem}

\begin{remark}
Stringy Hodge numbers  coincide with usual Hodge numbers 
in the smooth case, so $h^{p,q}_{st}(X_W)=h^{p,q}(X_W)$. However, the variety $Y_W$ is singular for $n\geq 11$, and the statement would fail without this correction. 
\end{remark}

\medskip
In the rest of this section we recall the definition of stringy Hodge numbers of singular varieties with log-terminal singularities following
\cite{Batyrev.1} and describe the  case of Zariski locally trivial log resolutions
which will be an important technical tool in our study of stringy Hodge numbers of $Y_W$.

\medskip
Let $Y$ be a singular variety with log-terminal singularities. Let $\pi:\widehat Y \to Y$ be a log resolution of $Y$, i.e. a proper birational morphism from a smooth variety $\widehat Y$ such that the exceptional divisor $\bigcup_{i=1}^kD_i$ has simple normal crossings. It is assumed that $Y$ is $\QQ$-Gorenstein, 
which allows us to compare the canonical classes
$$
K_{\widehat Y}\equiv \pi^*K_Y + \sum_{i=1}^{k}\alpha_iD_{i}
$$
to define the discrepancies $\alpha_i$. The discrepancies satisfy $\alpha_i>-1$ by log-terminality assumption (see \cite{CKM}).
Recall that for any variety  $W$ (not necessarily projective) we can define the Hodge-Deligne polynomial $E(W;u,v)$ which 
measures the alternating sum of dimensions  
of the $(p,q)$ components of the mixed Hodge structure on the cohomology 
of $W$ with compact support \cite{DKh}.
 For a possibly empty subset $J$ of $\{1,\ldots,k\}$ we define by 
$D_J^\circ$ the locally closed subset of $\widehat Y$ which consists of points $z\in \widehat Y$ that lie in $D_j$ if and only 
if $j\in J$.

\begin{definition}\label{def.Est}(\cite{Batyrev.1})
Stringy $E$-function of $Y$ is defined by 
$$
E_{st}(Y;u,v):=\sum_{J\subseteq \{1,\ldots , k\}} E(D_J^\circ;u,v) \prod_{j\in J} \frac {uv-1}{(uv)^{\alpha_j+1}-1}.
$$
Thus defined $E_{st}(Y;u,v)$ does not depend on the choice of the log resolution of $Y$, which justifies the notation.
It is in general only a rational function in fractional powers of $u,v$. 
However, in the case when $E_{st}(Y;u,v)$ is polynomial in $u,v$, we use 
$$
E_{st}(Y;u,v) = \sum_{p\geq 0,q\geq 0} (-1)^{p+q}h^{p,q}_{st}(Y) u^pv^q
$$
to define stringy Hodge numbers $h^{p,q}_{st}(Y)$.
\end{definition}

\begin{remark}
It is not clear under what conditions the stringy Hodge numbers exist. Even in three-dimensional Gorenstein case 
one can have $E$-functions with nontrivial denominators, see \cite{DR}. So 
the existence of stringy Hodge numbers of $Y_W$ is not known a priori. Rather, it is  a consequence of our calculation of its stringy $E$-function.
\end{remark}

\medskip
In a number of cases,  the log resolution of singularities $\widehat Y\to Y$ has an additional property of having the open strata $D_J^\circ$  form Zariski locally trivial fibrations over the corresponding
strata in $Y_W$.  This is the case when $Y$ has isolated singularities, but it also occurs more generally.
Notably, this  happens in the case of generic hypersurfaces and complete intersections in toric 
varieties,  as well as in the case considered in this paper. We discuss this phenomenon below.

\begin{definition}
We call a log resolution $\pi:\widehat Y\to Y$ as above \emph{Zariski locally trivial} if each $D_J^\circ$ is a Zariski locally trivial fibration over its image $\pi(D_J^\circ)$ in $Y$. 
\end{definition}

\begin{definition}\label{s}
Suppose that a singular variety $Y$ admits a Zariski locally trivial log resolution
$\pi:\widehat Y\to Y$. For a point $y\in Y$ define the local contribution of $y$ to $E_{st}(Y;u,v)$ to
be 
$$
S(y;u,v):=\sum_{J\subseteq \{1,\ldots , k\}} E(D_J^\circ\cap \pi^{-1}(y);u,v) \prod_{j\in J} \frac {uv-1}{(uv)^{\alpha_j+1}-1}.
$$
\end{definition}

\begin{remark}
Thus defined $S(y;u,v)$ is a constructible function on $Y$ with values in the field of rational functions in fractional powers of $u$ and $v$. Indeed, if $y_1$ and $y_2$ are such that the set of $J$ with $\pi(D_J^{\circ})$ that contain $y_1$ is the same as those that contain $y_2$, then $S(y_1;u,v)=S(y_2;u,v)$.
Moreover, this function is independent from the choice of a Zariski locally trivial resolution. This follows from
the usual argument that involves weak factorization theorem \cite{AKMW}. Last but not least, there holds
\begin{equation}\label{Zarlogtriv}
E_{st}(Y;u,v) = \sum_{i} E(Y_i;u,v) S(y\in Y_i;u,v)
\end{equation}
where $Y=\bigsqcup_i Y_i$ is the stratification of $Y$ into the sets on which $S$ is constant. This follows immediately 
from the multiplicativity of Hodge-Deligne $E$-functions for Zariski locally trivial fibrations, see \cite{DKh}.
\end{remark}

\section{Log resolutions and discrepancies of Pfaffian varieties}\label{sec.log}
In this section we use the classical spaces of complete skew forms to construct  log resolutions of Pfaffian varieties. We describe these spaces in detail. In particular, we calculate the discrepancies of the exceptional divisors, which are needed for the subsequent calculations of stringy $E$-functions. It is a rather delicate calculation based on the interplay between the spaces of complete skew forms on even and odd dimensional spaces.

\medskip
Let $V$ be a complex vector space of dimension $n\geq 3$. For now we do not assume that $n$ is odd.
 Consider 
the space $\PP \Lambda^2V^\dual$ of nontrivial skew forms on $V$ up to scaling. The loci of 
forms of corank $k\geq 0$ are locally closed smooth subvarieties of $\PP \Lambda^2V^\dual$ of 
codimension $\frac 12 k(k-1)$.
Note that the rank $n-k$ is always even. In particular, when $n$ is even, $\PP \Lambda^2V^\dual$ has 
an open stratum for $k=0$, a codimension one stratum 
for $k=2$ given by the vanishing of the Pfaffian of the form, as well as lower dimensional strata if $n$
is large enough. When $n$ is odd, there is an open stratum for $k=1$, a codimension three stratum 
for $k=3$, and possibly lower dimensional strata.

\medskip
For $n$ large enough, the closures of the strata in the above stratification of $\PP \Lambda^2V^\dual$ 
are singular. The space of complete skew forms provides a log resolution of this stratification. It is 
described in the following proposition.

\begin{proposition}\label{BT}
Consider the successive blowups of the loci of forms of rank $2$ in $\PP \Lambda^2V^\dual$,
then the proper preimage of the locus of forms of rank $4$, etc.
At each stage the center of the blowup is smooth, so all of the blowups are smooth. The resulting 
space of \emph{complete skew forms} is a smooth variety $\widehat{\PP \Lambda^2V^\dual}$ 
which parameterizes the (possibly trivial) flags
$$
0\subseteq F^0 \subset \cdots \subset F^l = V
$$
with $F^0$ of dimension $0$ if $n$ is even and $1$ if $n$ is odd, 
together with \emph{nondegenerate} forms $\CC w_i \in \PP \Lambda^2(F^{i+1}/F^i)^\dual$.
The map to $\widehat{\PP \Lambda^2V^\dual}\to {\PP \Lambda^2V^\dual}$ is given by 
interpreting a skew form on $F^{l}/F^{l-1}$ as a skew form on $F^l=V$.
\end{proposition}

\begin{proof} See \cite{Bertram, Thaddeus}. 
\end{proof}

\begin{remark}
It is worth pointing out that in the statement of Proposition \ref{BT} the length of the flag $l$ varies
from point to point. For a closed point in $\widehat{\PP \Lambda^2V^\dual}$ in general position we have $l=1$,
and the image in ${\PP \Lambda^2V^\dual}$ is a form of rank $n-3$.
\end{remark}

\begin{proposition}\label{disc}
The exceptional divisors of $\widehat{\PP \Lambda^2V^\dual}\to {\PP \Lambda^2V^\dual}$ 
are described by requiring that a subspace of certain dimension is present in the flag $F^\bullet$.
They form a simple normal crossing divisor on 
$\widehat{\PP \Lambda^2V^\dual}$.
The generic point of the divisor
that corresponds to the subspace of dimension $k$ 
maps to the generic point of the locus of
forms of corank $k$. The discrepancy of the corresponding divisor is $\frac 12k(k-1)-1$.
\end{proposition}

\begin{proof}
From the iterated blowup construction we see that the exceptional divisor is a simple normal
crossing divisor (see \cite{Bertram}). The description of it at the set-theoretic level is clear as well.
The discrepancies are calculated based on the codimension of the (smooth) locus of the corresponding
blowup.
\end{proof}

\medskip
A slight variation of this construction for odd $n$ produces a log resolution of the Pfaffian variety of degenerate forms on $V$.
\begin{proposition}
For odd $n$ consider the sequence of iterated blowups of (proper preimages) of loci of forms 
of rank $2$, $4$, and so on, up to $n-5$ in $\PP\Lambda^2 V^\dual$. Then consider the proper preimage
 $\widehat{Pf(V)}$
of the locus $Pf(V)$ of forms of corank at least $3$ on $V$. Then $\widehat{Pf(V)}\to Pf(V)$
is a log resolution of $Pf(V)$. Points of   $\widehat{Pf(V)}$ are given by collections of flags
$$
0\subset F^0 \subset \cdots \subset F^l = V
$$
with $F^0$ of dimension $3$ 
together with \emph{nondegenerate} forms $\CC w_i \in \PP \Lambda^2(F^{i+1}/F^i)^\dual$.
The exceptional divisors $D_j$ of the map $\widehat{Pf(v)}\to Pf(V)$ are indexed by $j=3,\ldots, \frac {n-1}2$
and are characterized by the existence of a subspace of dimension $2j-1$ in the above flag.
\end{proposition}

\begin{proof}
The smoothness is part of the statement of Proposition \ref{BT}. The set-theoretic description of the space is also clear.
\end{proof}

\begin{remark}
We can also identify the space $\widehat{Pf(V)}$ with the \emph{relative} space of complete skew forms
(see \cite{Bertram}) on the rank $n-3$ universal quotient bundle $Q_3$ over the Grassmannian $G(3,V)$. Indeed, the map to 
$\PP \Lambda^2 Q_3^\dual$ is given by sending the form with the above data to $F^0$ and
the skew form on $F^l/F^0=V/F^0$.
\end{remark}

The canonical class of $Pf(V)$ has been calculated in \eqref{canPf}.
We will now calculate the discrepancies of the map $\pi:\widehat{Pf(V)}\to Pf(V)$.
\begin{theorem}\label{discpf}
We have the following relation in the Picard group of $\widehat{Pf(V)}$.
$$
K_{\widehat{Pf(V)}} = \pi^*K_{Pf(V)} + \sum_{j=3}^{\frac 12(n-1)} (2j^2-5j+1) D_j.
$$
In particular, $Pf(V)$ has terminal Gorenstein singularities.
\end{theorem}

\begin{proof}

Consider the variety $Z=\widehat{Pf(V)}\times_{G(3,V)}Fl(2,3,V)$ 
which parametrizes elements of $\widehat{Pf(V)}$ together with a choice of 
a dimension two subspace inside the tautological space for the corresponding point in $G(3,V)$.
This $Z$ is simply the space of complete skew forms on the universal quotient bundle $Q_{3,Fl}$
over the partial flag variety $Fl(2,3,V)$.
Points of $Z$ are given by flags 
$$
0\subset F^{-1}\subset F^0\subset \cdots \subset F^l=V
$$
with $\dim F^{-1}=2$, $\dim F^0=3$ and the nondegenerate forms 
$\CC w_i \in \PP \Lambda^2(F^{i+1}/F^i)^\dual$ for $i\geq 0$.
Of course, $Z$ is a $\PP^2$ bundle over $\widehat{Pf(V)}$.

\smallskip
Note that $Z$ maps to $G(2,V)$. In fact, it is clear that this map passes through  the space
$
\PP \Lambda^2 Q_2^\dual
$
of skew forms on the universal quotient bundle over the Grassmannian $G(2,V)$. Moreover, we can view 
$Z$ as the relative space of complete skew forms on $Q_2$ over $G(2,V)$. We have the following
commutative diagram.
$$
\begin{array}{ccccccc}
&&\widehat{Pf(V)}& \longleftarrow & Z&&\\
&&\downarrow &                 & \downarrow&\searrow&    \\
Pf(V)&\longleftarrow&\PP \Lambda^2 Q_{3}^\dual&\longleftarrow&\PP \Lambda^2 Q_{3,Fl}^\dual&\longrightarrow&\PP \Lambda^2 Q_2^\dual \\
&&\downarrow &                 & \downarrow& & \downarrow   \\
&&G(3,V)&\longleftarrow&Fl(2,3,V)&\longrightarrow&G(2,V)
\end{array}
$$
There are also natural morphisms $\PP \Lambda^2 Q_2^\dual\to Pf(V)$ and $Pf(V)\to \PP \Lambda^2 V^\dual$
which are not depicted in the above diagram but which commute with all of the above morphisms.

\begin{remark}\label{rem}
The natural morphism $\PP \Lambda^2 Q_{3,Fl}^\dual\to \PP \Lambda^2 Q_2^\dual$ is birational.
Indeed, a (maximum) rank $n-3$ form on a fiber of $Q_2$ has a one-dimensional kernel, which defines the 
three dimensional subspace in the flag. However, even though both varieties are smooth, the exceptional locus
$E$ of (relative) forms of positive corank is not a smooth divisor. The blowup locus in 
$\PP \Lambda^2 Q_2^\dual$ consists of (relative) skew forms of corank at least three and 
is generically smooth and of codimension three.
%Thus, the discrepancy of $E$  for   $\PP \Lambda^2 Q_{3,Fl}^\dual\to \PP \Lambda^2 Q_2^\dual$ is equal to % two.
\end{remark}

\smallskip
There are divisors $E_j$ on $Z$ defined as the loci of the complete skew forms that have
a filtration subspace of dimension $2j-1$. The index $j$ ranges from $j=3$ to $j=\frac 12(n-1)$.
We observe that $E_j$ are the exceptional divisors of the birational morphisms 
to $\PP \Lambda^2 Q_{3,Fl}^\dual$ and $\PP \Lambda^2 Q_2^\dual$. They are also preimages of the exceptional divisors $D_j$ of $\widehat {Pf(V)}\to Pf(V)$. 

\smallskip 
The discrepancies of the map $\pi_1:Z\to \PP \Lambda^2 Q_{3,Fl}^\dual$ can be calculated 
by Proposition \ref{disc} and we get
\begin{equation}\label{1}
\begin{split}
K_Z =& \pi_1^* K_{ \PP \Lambda^2 Q_{3,Fl}^\dual} + \sum_{j=3}^{\frac 12(n-1)} ((j-2)(2j-5)-1)E_j
%\\=&\pi_1^* K_{ \PP \Lambda^2 Q_{3,Fl}^\dual} + \sum_{j=3}^{\frac 12(n-1)} (2j^2-9j+9)E_j
\\
=&\pi_3^* K_{Fl(2,3,V)} -\pi_3^*c_1(\Lambda^2 Q_{3,Fl}^\dual)- {\rm rank} (\Lambda^2 Q_{3,Fl}^\dual)\xi
+ \sum_{j=3}^{\frac 12(n-1)} (2j^2-9j+9)E_j
\\
=&\pi_3^* K_{Fl(2,3,V)} -\pi_3^*c_1(\Lambda^2 Q_{3,Fl}^\dual)- \frac 12(n-3)(n-4)\xi
+ \sum_{j=3}^{\frac 12(n-1)} (2j^2-9j+9)E_j
\\
=&\pi_3^* K_{Fl(2,3,V)} +(n-4)\pi_3^*c_1( Q_{3,Fl})- \frac 12(n-3)(n-4)\xi
+ \sum_{j=3}^{\frac 12(n-1)} (2j^2-9j+9)E_j
\end{split}
\end{equation}
for $\pi_3:Z\to Fl(2,3,V)$. Here $\xi$ is the pullback to $Z$ of the first chern class of the universal
quotient bundle on $ \PP \Lambda^2 Q_{3,Fl}^\dual$ which is also the same as the pullback
of the hyperplane class via $Z\to \PP\Lambda^2(V)$.
Note that the discrepancy of $E_3$ is $0$. Indeed,  $\pi_1(E_3)$ is already a divisor, namely
the exceptional divisor of Remark \ref{rem}.

\smallskip
Similarly, the morphism $\pi_2:Z\to \PP \Lambda^2 Q_2^\dual$ is a relative construction of 
the space of complete forms, obtained by blowing up proper preimages of 
 loci of forms of corank $2j-3$ in the odd dimensional spaces, which gives
\begin{equation}\label{2}
\begin{split}
K_Z =& \pi_2^* K_{ \PP \Lambda^2 Q_{2}^\dual} + \sum_{j=3}^{\frac 12(n-1)} ((2j-3)(j-2)-1)E_j
%=\pi_2^* K_{ \PP \Lambda^2 Q_{2}^\dual} + \sum_{j=3}^{\frac 12(n-1)} (2j^2-7j+5)E_j
\\
=&
\pi_4^* K_{G(2,V)} -\pi_4^*c_1(\Lambda^2 Q_{2}^\dual)- {\rm rank} (\Lambda^2 Q_{2}^\dual)\xi
+ \sum_{j=3}^{\frac 12(n-1)} (2j^2-7j+5)E_j
\\
=&
\pi_4^* K_{G(2,V)}- \pi_4^*c_1(\Lambda^2 Q_{2}^\dual)- \frac 12(n-2)(n-3)\xi
+ \sum_{j=3}^{\frac 12(n-1)} (2j^2-7j+5)E_j
\\
=&-3\pi_4^*c_1(Q_2)- \frac 12(n-2)(n-3)\xi
+ \sum_{j=3}^{\frac 12(n-1)} (2j^2-7j+5)E_j
\end{split}
\end{equation}
for $\pi_4:Z\to G(2,V)$.

\smallskip 
We take a linear combination of the above equations \eqref{1} and \eqref{2}  with coefficients $-1$ and $2$ to get
\begin{equation}\label{3}
K_Z = 
-\pi_3^* K_{Fl(2,3,V)} -(n-4)\pi_3^*c_1(Q_{3,Fl}) -6 \pi_4^*c_1( Q_{2})
- \frac 12 n(n-3)\xi
+ \sum_{j=3}^{\frac 12(n-1)} (2j^2-5j+1)E_j.
\end{equation}
On the other hand, for $\pi:\widehat{Pf(V)}\to Pf(V)$ we have 
$$
K_{\widehat{Pf(V)}}=\pi^*K_{Pf(V)}+\sum_{j}\alpha_j D_j
=-\frac 12n(n-3)\pi^*\xi +\sum_{j}\alpha_j D_j
$$
by \eqref{canPf}
and thus for $\mu:Z\to \widehat{Pf(V)}$ and $\pi_5:Z\to G(3,V)$ we have
\begin{equation}\label{4}
\begin{split}
K_Z =&
 \mu^* K_{\widehat{Pf(V)}} + \pi_3^*K_{Fl(2,3,V)}-\pi_5^*K_{G(3,V)}
\\
=&
-\frac 12 n(n-3)\xi+\sum_{j}\alpha_j E_j+\pi_3^*K_{Fl(2,3,V)}-\pi_5^*K_{G(3,V)}.
\end{split}
\end{equation}
When we compare the formulas \eqref{3} and \eqref{4} for $K_Z$ we get:
$$
0=-2\pi_3^* K_{Fl(2,3,V)} -(n-4)\pi_3^*c_1(Q_{3,Fl}) -6 \pi_4^*c_1( Q_{2})
+\pi_5^*K_{G(3,V)}
+ \sum_{j=3}^{\frac 12(n-1)} (2j^2-5j+1-\alpha_j)E_j.
$$
Thus, to finish the proof of Theorem \ref{discpf},
 it remains to verify that 
$$
0=-2\pi_3^* K_{Fl(2,3,V)} -(n-4)\pi_3^*c_1(Q_{3,Fl}) -6 \pi_4^*c_1( Q_{2})
+\pi_5^*K_{G(3,V)}
$$
in the Picard group of $Z$. All of the ingredients of this formula are pullbacks 
from the partial flag variety $Fl(2,3,V)$, and the statement follows from
\begin{equation}\label{left}
0=-2 K_{Fl(2,3,V)} - (2n-4)c_1(Q_{3,Fl}) - 6 c_1( Q_{2,Fl})
\end{equation}
in the Picard group of $Fl(2,3,V)$, which we verify below.

\smallskip
The tangent bundle to $Fl(2,3,V)$ fits into a short exact sequence with the bundles 
$Hom(T_{3,Fl},Q_{3,Fl})$ and $Hom(T_{2,Fl},T_{3,Fl}/T_{2,Fl})$,
where $T$ denotes the appropriate tautological subbundles. Thus, we have
$$-K_{Fl(2,3,V)} = n c_1(Q_{3,Fl})+c_1(Hom(T_{2,Fl},T_{3,Fl}))-c_1(Hom(T_{2,Fl},T_{2,Fl}))$$
$$
=n c_1(Q_{3,Fl})-3c_1(T_{2,Fl})+2c_1(T_{3,Fl}) = 
(n-2) c_1(Q_{3,Fl})+3c_1(Q_{2,Fl}).
$$
This proves \eqref{left} and finishes the proof of Theorem \ref{discpf}.
\end{proof}

\begin{remark}
We get a log resolution of $\widehat Y_W \to Y_W$ by taking a complete intersection of $\widehat{Pf(V)}$ 
by $\PP W \subset \PP\Lambda^2V^\dual$.
For the generic choice of $W$, this resolution 
has the exceptional divisors $D_j\cap \widehat Y_W$ with the same discrepancies. However, there are considerably fewer of them, since the 
codimension of the image of $D_j$ in $Pf(V)$ is quadratic in $j$, so most $D_j$ have an empty intersection with the preimage of $\PP W$.
\end{remark}

\section{Stringy $E$-functions of Pfaffian varieties}\label{sec4}
The goal of this section is to calculate the stringy $E$-functions of Pfaffian subvarieties
 $Pf(\CC^{2r+1})\subset \PP\Lambda^2(\CC^{2r+1})$. We do so by induction on $r$. Our main result 
is a remarkably simple formula of 
Theorem \ref{stpf}.

\medskip
We  start with some results on the usual $E$-functions of the loci of skew forms of fixed rank.

\begin{definition}
For $i\geq 1$ we denote by $e_{2i}=e_{2i}(u,v)$ the $E$-function of the variety of 
\emph{nondegenerate} skew forms on $\CC^{2i}$, up to scaling.  For any $0\leq k\leq n$ 
we denote by $g_{k,n}(u,v)$  the $E$-functions of the 
Grassmannian $G(k,n)$. 
\end{definition}

\begin{remark}
It is possible to write explicit formulas for $e_{2i}(u,v)$ and $g_{k,n}(u,v)$.
but we will not use them in this section. The formula for $g_{k,n}(u,v)$ will be given and used in the Appendix to help deal with the more
general Pfaffian double mirror conjecture.
\end{remark}

\begin{proposition}\label{relg}
For $0\leq i\leq r$ there holds 
$$
g_{2i,2r}(u,v) = g_{2i,2r+1}(u,v) \left(\frac {(uv)^{2r-2i+1}-1}{(uv)^{2r+1}-1}\right).
$$
\end{proposition}

\begin{proof}
The partial flag variety $Fl(2i,2r,2r+1)$ is a Zariski locally trivial fibration with fiber $G(2i,2r)$ over $\PP^{2r}$.
It is also a Zariski locally trivial fibration over $G(2i,2r+1)$ with fiber $\PP^{2r-2i}$. It remains to recall that 
$E(\PP^k)=\frac {(uv)^{k+1}-1}{uv-1}$.
\end{proof}

We now observe two relations among $e_{2i}$ and the $E$-functions of Grassmannians.
\begin{proposition}\label{oddeven}
The following identities hold for any $r\geq 1$ as functions of $u,v$.
\begin{equation}\label{even}
\sum_{i=1}^{r} e_{2i}g_{2i,2r} = \frac{(uv)^{r(2r-1)}-1}{uv-1}
\end{equation}
\begin{equation}\label{odd}
\sum_{i=1}^{r} e_{2i}g_{2i,2r+1} = \frac{(uv)^{r(2r+1)}-1}{uv-1}
\end{equation}
\end{proposition}

\begin{proof}
To prove the first identity observe that the space of skew forms $\PP\Lambda^2(\CC^{2r})$ is 
stratified by the rank $2i$ of the form for $1\leq i\leq r$. By considering the kernels of the forms of rank $2i$,
we deduce that the aforementioned $i$-th stratum is a
fibration over $G(2r-2i,2r)$ with the fiber isomorphic to the space of nondegenerate skew forms 
on $\CC^{2i}$. To show that this fibration is  Zariski locally trivial, let us view $G(2i-2r,\CC^{2r})$ as 
the space of full rank $(2r-2i,2r)$ matrices up to left action of $GL(2r-2i,\CC)$.
Then consider the subgroup of $\PP GL(2r,\CC)$ of matrices
$$
\left(
\begin{array}{cc}
I_{2r-2i} & M_{2r-2i,2i}
\\
0& I_{2i}
\end{array}
\right)
$$
whose right action on the subspace of dimension $2i-2r$ in $G(2r-2i,\CC^{2r})$
generated by first $2i-2r$ basis vectors is identifies it with an open Schubert cell.
Thus, we can use the action of this group to trivialize the fibration over this Schubert cell.
Then we cover the Grassmannian by translates of this cell and use the conjugate subgroups.
Finally we use $g_{2r-2i,2r}=g_{2i,2r}$. The second identity is proved similarly.
\end{proof}

\smallskip
The following observation is key to calculating stringy $E$-functions of Pfaffian varieties.
\begin{proposition}\label{sum}
For any $r\geq 1$ there holds
$$
\sum_{i=1}^{r-1}\left(\frac{(uv)^{2r-2i}-1}{(uv)^2-1}\right) e_{2i}g_{2i,2r+1} =
 \frac{((uv)^{2r}-1)((uv)^{2r^2-r-1}-1)}{((uv)^2-1)(uv-1)}.
$$
\end{proposition}

\begin{proof}
We can  change the index of summation to $i=1,\ldots,r$, since the $i=r$ term is identically zero.
We combine the results of Propositions \ref{relg}  and \ref{oddeven} as follows.
By Proposition \ref{relg} we have
$$
\left(\frac{(uv)^{2r-2i}-1}{(uv)^2-1}\right)g_{2i,2r+1} 
%=
%\left(\frac{(uv)^{2r-2i+1}-uv}{(uv)((uv)^2-1)}\right)g_{2i,2r+1} 
%$$
%$$
%=\left(\frac{(uv)^{2r-2i+1}-1}{(uv)((uv)^2-1)}\right)g_{2i,2r+1} 
%-
%\left(\frac{uv-1}{(uv)((uv)^2-1)}\right)g_{2i,2r+1} 
%$$
%$$
= \frac {1}{uv((uv)^2-1)}\left(
((uv)^{2r+1}-1)
g_{2i,2r} -(uv-1)g_{2i,2r+1}
\right).$$
Then we use \eqref{even}  and \eqref{odd}
to get
$$
\sum_{i=1}^{r}\left(\frac{(uv)^{2r-2i}-1}{(uv)^2-1}\right) e_{2i}g_{2i,2r+1}
=\frac {1}{uv((uv)^2-1)}\left(
((uv)^{2r+1}-1)
 \frac{(uv)^{r(2r-1)}-1}{uv-1} \right.
 $$
 $$\left.-
 (uv-1) \frac{(uv)^{r(2r+1)}-1}{uv-1} 
\right)= \frac{((uv)^{2r}-1)((uv)^{2r^2-r-1}-1)}{((uv)^2-1)(uv-1)}.
$$
\end{proof}

The main result of this section is the formula for the stringy $E$-function of the Pfaffian variety 
$Pf(\CC^{2r+1})$ of skew forms of rank $\leq 2r-2$ in $\PP\Lambda^2(\CC^{2r+1})$.
\begin{theorem}\label{stpf}
For any $r\geq 2$ the stringy $E$-function of the Pfaffian variety $Pf(\CC^{2r+1})$ is given by
$$
E_{st}(Pf(\CC^{2r+1});u,v)=\frac{((uv)^{2r}-1)((uv)^{2r^2-r-1}-1)}{((uv)^2-1)(uv-1)}.
$$
\end{theorem}

\begin{proof}
We will argue by induction on $r$. The case $r=2$ is straightforward, since $Pf(\CC^5)$ is
a smooth variety isomorphic to $G(3,5)$, whose cohomology is well known.

\smallskip
We now assume  the result of this theorem
for $Pf(\CC^{2k+1})$ for all $k<r$ and consider a vector space $V=\CC^{2r+1}$.
Observe that the log resolution 
$\widehat{Pf(V)}$ of
$Pf(V)\subset \PP\Lambda^2 V^\dual$ considered in Section \ref{sec.log} is naturally 
stratified by specifying various choices 
of the subspaces for the partial complete skew forms.
Specifically, the strata are in bijections with subsets $I$ of the set of odd integers $\{3,\ldots, 2r+1\}$
that include $3$ and $2r+1$. The contribution of the said stratum to the stringy $E$-function of $Pf(V)$ 
as defined by Definition \ref{def.Est} is the product of the $E$-polynomial of the stratum with 
$$
\prod_{3\leq j\leq r,~2j-1\in I} \left(
\frac {uv-1}{(uv)^{2j^2-5j+2}-1}
\right).
$$

\smallskip
Each of these strata is  a Zariski locally
trivial fibration over the locus of forms of rank $2i$ on $V$ where $2i$ is the codimension of the largest
proper subspace in the flag. Indeed, observe that the preimage of a point $\CC w\in Pf(V)$ of 
rank $2i$ is naturally isomorphic to the space of partial skew forms on ${\rm Ker}(w)$ so we can locally
trivialize the fibration by using the subgroup trick of the proof of Proposition \ref{oddeven}. Notice now that
this fiber is the log resolution of the lower-dimensional Pfaffian $Pf(\CC^{2r-2i+1})$! Moreover, we can relate the contributions 
of the strata of the resolution of $Pf(\CC^{2r-2i+1})$ to that of $Pf(V)$ by observing that they lie in an additional
divisor $D_j$ with $j={r-i}+1$, but otherwise lie in the same set of divisors. So there is an additional factor of 
$$
\frac{uv-1}{(uv)^{2j^2-5j+2}-1} 
=
\frac{uv-1}{(uv)^{2(r-i)^2-(r-i)-1}-1} 
$$
when one compares the contribution to $Pf(V)$ as opposed to the contribution to $Pf(\CC^{2r-2i+1})$.
Thus, the $S$-function in the sense of Definition \ref{s} is equal to
$$S(\CC w;u,v)= E_{st}(Pf(\CC^{2r-2i+1})\left(
\frac{uv-1}{(uv)^{2(r-i)^2-(r-i)-1}-1} 
\right)
=
\frac{(uv)^{2(r-i)}-1}{(uv)^{2}-1}
$$
by the induction assumption.
 
\smallskip
We  now use the formula \eqref{Zarlogtriv} from Section \ref{sec2}.
The locus of forms of rank $2i$ in $\PP\Lambda^2V^\dual$ is a Zariski locally trivial fibration over
$G(2r+1-2i,V)$ with fibers the spaces of nondegenerate skew forms on $\CC^{2i}$ up to scaling. Thus,
this locus 
 has $E$-polynomial $e_{2i}g_{2i,2r+1}$.
Thus, the contribution of the strata in $\widehat{Pf(V)}$ that lie over the locus of forms of rank $2i$ 
is equal to 
$$
e_{2i}g_{2i,2r+1}
\left(
\frac{(uv)^{2(r-i)}-1}{(uv)^{2}-1}
\right).
$$
The index $i$ runs from $1$ to $r-1$, but the contribution of the open stratum
$i=r-1$ needs to be considered separately,
because $Pf(\CC^3)$ is not defined, and the intermediate formula does not make sense. However,
this contribution is easily seen to be $e_{2r-2}g_{3,2r+1}$, so the final formula works  for $i=r-1$ as well.
We now have 
$$
E_{st}(Pf(\CC^{2r+1}))=\sum_{i=1}^{r-1}e_{2i}g_{2i,2r+1}
\left(\frac{(uv)^{2(r-i)}-1}{(uv)^{2}-1}\right)
$$
and it remains to use Proposition \ref{sum}.
\end{proof}

\begin{remark}\label{remspace}
Observe that the usual $E$-function of $Pf(\CC^{2r+1})$ is a rather complicated polynomial in $u,v$
given by 
$E(\PP^{2r^2+r-1})-e_{2r}E(\PP^{2r})$.
It is satisfying, even if somewhat expected, that the stringy $E$-function of the Pfaffian variety is 
simpler than its usual $E$-function. This is yet another justification for the use of $E_{st}$. We intend
to look for a natural vector space (or rather a flat family of vector spaces) to serve as stringy cohomology of the 
Pfaffian, similar to the toric singularities case of \cite{torloc}.
\end{remark}

\section{Comparison of stringy $E$ functions}\label{proof}
In this section we  prove our main result Theorem \ref{main} that compares the (stringy) $E$-functions of the double mirror
Calabi-Yau manifolds $X_W$ and $Y_W$. The main idea is to reduce the calculation of the 
$E$-function of $X_W$ to that of its Cayley hypersurface $H$ which we define below. Then we consider the projection of $H$
onto $\PP W$ and look at the fibers of that projection over different loci in $Y_W$.

\medskip
Recall that we have subspace $W$ of dimension $n$ in $\Lambda^2 V^\dual$, the Grassmannian complete intersection $X_W$ in $G(2,V)$ and the Pfaffian locus $Y_W$ in $\PP W$. Consider the 
Cayley hypersurface $H\subset G(2,V)\times \PP W$ which consists of $(T_2,w)$ with $w\Big\vert_{T_2}=0$.

\medskip
We first connect the $E$-function of $X_W$ with that of $H$. The projection $H\to G(2,V)$ has 
fibers $\PP^{n-2}$ over $T_2\not\in X_W$ and fibers $\PP^{n-1}$ over $T_2\in X_W$. Moreover, on
these loci the fibration is Zariski locally trivial, since it is a projectivization of a vector bundle. Therefore,
\begin{align*}
E(H;u,v) &=\Big(E(G(2,V))-E(X_W)\Big) E(\PP^{n-2}) + E(X_W)E(\PP^{n-1}) 
\\
&=E(G(2,V)) \frac{(uv)^{n-1}-1}{uv-1} +E(X_W)(uv)^{n-1}.
\end{align*}
The $E$-function of $G(2,V)$ can be calculated by realizing $G(2,V)$ as the base of the Zariski locally trivial
fibration of the space of ordered pairs of linearly independent vectors in $V$. The fiber is $GL(2,\CC)$ with 
$E$-polynomial $((uv)^2-1)((uv)^2-uv)$. The total space has $E$ polynomial $((uv)^n-1)((uv)^n-uv)$
(first pick one nonzero vector, then pick a vector not in its span). Thus we get
\begin{equation}\label{chigr}
E(H;u,v) = E(X_W;u,v)(uv)^{n-1} +\Big(\frac{(uv)^{n-1}-1}{(uv)-1}\Big)\Big(\frac {((uv)^n-1)((uv)^n-uv)}{((uv)^2-1)((uv)^2-uv)}\Big).
\end{equation}

\medskip
We now consider the projection of $H$ to the second factor $\PP W$. 
While we do not know whether this fibration is Zariski locally trivial over each locus of 
the forms $w\in\PP W$ of $\dim {\rm Ker}(w) = 2k+1$ for $k=0,1,\ldots$, we will show that the $E$-function
is still multiplicative on the fibers. Let us denote these strata of $\PP W$
by $P_k$. For example, $Y_W$ is the closure of $P_1$. The fiber of the projection $H\to W$ over the stratum
$P_k$ is the hypersurface in $G(2,V)$ given by $w=0$ for $\dim {\rm Ker}(w) = 2k+1$. Let us calculate the 
$E$-function of this hypersurface.

\begin{lemma}\label{fk}
The $E$ function of the fiber $F_k$ of $H\to \PP W$ over $w$ of corank $2k+1$ is given by
$$
E(F_k;u,v) = \Big( \frac {(uv)^{2k}-1}{(uv)^2-1}\Big)(uv)^{n-1} + \frac 1{uv+1}
\Big(\frac {(uv)^{n-1}-1}{uv-1}\Big)^2.
$$
\end{lemma}

\begin{proof}
As is in the Grassmannian calculation, we consider the Zariski locally  trivial $GL(2,\CC)$ fibration $B_k$ over this
$F_k$ given by linearly independent $(v_1,v_2)$ in $V$ with the property 
$w(v_1,v_2)=0$.  The part of $B_k$ with 
$v_1\in {\rm Ker}(w)$ is fibered over ${\rm Ker}(w)-\{0\}$ with fiber $\CC^n-\CC$, because $v_2$ can be picked
arbitrarily. If $v_1\not\in {\rm Ker}(w)$ then the choices for $v_2$ are $\CC^{n-1}-\CC$ (again, the fibration is 
Zariski locally trivial). This gives
\begin{align*}
E(B_k) =& ((uv)^{2k+1}-1)((uv)^n-uv) + ((uv)^n-(uv)^{2k+1})((uv)^{n-1}-uv) 
\\=&
(uv)^{2k+n+1}-(uv)^n-(uv)^{2k+2}+uv+(uv)^{2n-1}-(uv)^{n+1}
-(uv)^{2k+n}
\\&
+(uv)^{2k+2}=
((uv)^{2k}-1)((uv)^{n+1}-(uv)^n) + uv((uv)^{n-1}-1)^2
\end{align*}
and 
$$
E(F_k;u,v) = \frac {E(B_k)}{((uv)^2-1)((uv)^2-uv)} = \Big( \frac {(uv)^{2k}-1}{(uv)^2-1}\Big)(uv)^{n-1} + \frac 1{uv+1}
\Big(\frac {(uv)^{n-1}-1}{(uv)-1}\Big)^2.
$$
\end{proof}

\begin{lemma}\label{fk1}
Let $H_k$ be the preimage in $H$ of the locus $P_k$. Then there holds
$$
E(H_k)=E(P_k)E(F_k).
$$
\end{lemma}

\begin{proof}
We consider the frame bundle $\tilde H_k$ over $H_k$ whose fiber over a point $(T_2,w)\in H$ is the space of all bases of $T_2$. Thus points in $\tilde H_k$ encode triples $(v_1,v_2,w)$ where $v_1$ and $v_2$ are linearly independent,
$w\in P_k$ and $w(v_1,v_2)=0$. 
The map $\tilde H_k\to P_k$ has fibers isomorphic to $B_k$. 
As in the proof of Lemma \ref{fk}, the space $\tilde H_k$ 
can be further subdivided into two subsets depending on whether $v_1\in Ker(w)$ or $v_1\not\in Ker(w)$. For each of the subspaces,
we get an iterated structure of Zariski locally trivial fibration over $P_k$ and therefore
$$
E(\tilde H_k)=E(P_k)E(B_k)
$$
which implies the claim of the lemma.
\end{proof}

We are now able to calculate the $E$ function of $H$ using the projection to $\PP W$:
\begin{align*}
E(H;u,v) =& \sum_{k\geq 0} E(P_k)E(F_k) 
=\sum_{k\geq 0} E(P_k)  \Big( \frac {(uv)^{2k}-1}{(uv)^2-1}\Big)(uv)^{n-1} \\
&+
\sum_{k\geq 0}E(P_k) \frac 1{uv+1}
\Big(\frac {(uv)^{n-1}-1}{uv-1}\Big)^2
\\
=&\sum_{k\geq 0} E(P_k)  \Big( \frac {(uv)^{2k}-1}{(uv)^2-1}\Big)(uv)^{n-1}
+E(\PP W)\frac 1{uv+1}
\Big(\frac {(uv)^{n-1}-1}{uv-1}\Big)^2
\\=&\sum_{k\geq 0} E(P_k)  \Big( \frac {(uv)^{2k}-1}{(uv)^2-1}\Big)y^{n-1}
+\frac 1{uv+1}\Big(\frac {(uv)^{n}-1}{uv-1}\Big)\Big(\frac {(uv)^{n-1}-1}{uv-1}\Big)^2.
\end{align*}

We compare the above equation with \eqref{chigr} to get a simple formula
\begin{equation}\label{almost}
E(X_W;u,v) = \sum_{k\geq 0} E(P_k;u,v)  \Big( \frac {(uv)^{2k}-1}{(uv)^2-1}\Big).
\end{equation}

\begin{remark}
The contribution of $k=0$ term is zero, so
the right hand side of the above formula involves summation over the loci 
$P_k,k\geq 1$. For example, for $n\leq 9$ we have $P_{\geq 2}=\emptyset$ and $Y_W=P_1$
is smooth, so we immediately get the equality $E(X_W)=E(Y_W)$. 
\end{remark}

\begin{remark}
The above calculations can be performed in the Grothendieck ring of varieties over $\CC$. This idea was explored
in \cite{Bzd} to show that the class of affine line is a zero divisor in the Grothendieck ring.
\end{remark}

We are now ready to prove the first major result of this paper, Theorem \ref{main}.
\begin{proof}(of Theorem \ref{main})
The statement of equality of Hodge numbers is equivalent to 
$$E(X_W;u,v)=E_{st}(Y_W;u,v).$$
Observe that $Y_W$ is a transversal
complete intersection of the Pfaffian variety of degenerate forms on $V$. Therefore, the log resolution 
of $Y_W$ can be obtained by taking a complete intersection in $\widehat{Pf(V)}$. Different strata 
of the resolution are therefore Zariski locally trivial fibrations over $P_k$ for appropriate $k$. 
As in the proof of Theorem \ref{stpf}, we observe that the fibers over $P_k$ for $k\geq 2$ are log resolutions 
of the Pfaffian varieties $Pf(\CC^{2k+1})$, and  that the strata are now counted 
with the extra factor
$$
\frac{uv-1}{(uv)^{2k^2-k-1}-1} 
$$
due to the additional divisor $D_{k+1}$. 
Thus the contribution of the stratum $P_k$ to $E_{st}(Y_W;u,v)$ is
$$
E(P_k;u,v) E_{st}(Pf(\CC^{2k+1});u,v) \left(
\frac{uv-1}{(uv)^{2k^2-k-1}-1} 
\right)
=
E(P_k;u,v) \left(
\frac{(uv)^{2k}-1}{(uv)^{2}-1} 
\right)
$$
where we used Theorem \ref{stpf} to calculate the stringy $E$-function of $Pf(\CC^{2k+1})$.
The same formula applies for $k=1$, because the log resolution is the isomorphism over this locus.
Thus we have 
$$
E_{st}(Y_W;u,v) = \sum_{k\geq 1} E(P_k;u,v) 
\frac{(uv)^{2k}-1}{(uv)^{2}-1}
$$
which equals 
$E(X_W)$  in view of \eqref{almost}.
\end{proof}

\section{Even-dimensional case}\label{sec.even}
There is a similar, although slightly less appealing correspondence between complete intersections 
in the Grassmannian $G(2,V)$ and a Pfaffian variety in the case of even $n$. The double mirrors in this case are only so in some not particularly clear generalized sense. In particular, they are not of the same dimension,
so the relation of the Hodge numbers needs to be corrected. Remarkably, we see that one needs to somehow modify the definition of stringy Hodge numbers in order to get an analog of Theorem \ref{main} in the case
of even $n$.

\medskip
\begin{definition}
Let $n\geq 4$ be an even integer. Let $V$ be a vector space of dimension $n$ and let $W$ be a generic subspace of dimension $n$ in $\Lambda^2 V^\dual$. We define $X_W$ and $Y_W$ as follows.
\begin{itemize}
\item $X_W\subset G(2,V)$ is the locus of $T_2\subset V$ with $w\Big\vert_{T_2} = 0$ for all $w\in W$.
\item $Y_W$ is the hypersurface in $\PP W$ of skew forms $\CC w$ of positive corank. 
\end{itemize}
\end{definition}

\begin{remark}
For $n\geq 6$ we see that $X_W$ is a smooth Calabi-Yau variety of dimension $n-4$. It is a union of two points for $n=4$. For any $n\geq 4$ the variety $Y_W$ is a hypersurface of dimension $n-2$ and degree $\frac n2$
given by the vanishing of a single Pfaffian. It is smooth for $n\leq 6$ and has Gorenstein singularities for 
larger $n$.
\end{remark}

\begin{remark}
We still want to view $X_W$ and $Y_W$ as double mirror to each other in some generalized sense. There have been examples of generalized mirror symmetry,
for example in the setting of rigid Calabi-Yau varieties see \cite[Section 5]{BB}, but they have not been studied systematically. See also \cite[Conjecture 4.4]{Kuznetsov2} and \cite{IM}.
\end{remark}

\smallskip
To understand what type of relation between the $E$-functions of $X_W$ and $Y_W$ one might expect,
we derive the analog of equation \eqref{almost}.

\begin{proposition}\label{almost1}
We denote by $P_k$ the locally closed subvariety of  $\PP \Lambda^2 V^\dual$ 
of forms of corank $2k$.  In particular, the closure of $P_1$ is $Pf(V)$.
For any $n\geq 4$ there holds
$$
(uv)E(X_W;u,v) = \sum_{k\geq 0}\Big( \frac {(uv)^{2k}-1}{(uv)^2-1}\Big) E(P_k;u,v) 
 - \frac{(uv)^n-1}{(uv)^2-1}.
$$
\end{proposition}

\begin{proof}
As before, we have the Cayley hypersurface $H\subset G(2,V)\times \PP W$ whose $E$ function is related to that of $X_W$ by 
$$
E(H;u,v) = E(X_W;u,v)(uv)^{n-1} +\Big(\frac{(uv)^{n-1}-1}{(uv)-1}\Big)\Big(\frac {((uv)^n-1)((uv)^n-uv)}{((uv)^2-1)((uv)^2-uv)}\Big). 
$$
The projection of $H$ onto $\PP W$ is a disjoint union of fibrations with fibers 
$F_k$ over the loci $P_k$ of forms of corank $2k$ for $k\geq 0$. 
The $E$-function of the fiber $F_k$ and the contribution of the locus $P_k$ 
are calculated similarly to Lemmas \ref{fk} and \ref{fk1} (the notation $F_k$ has a slightly different meaning now due to
different corank) as 
$$
E(F_k;u,v)=
% \frac {1}{(t^2-t)(t^2-1)}\Big(
%(t^n-t^{2k})(t^{n-1}-t) + (t^{2k}-1)(t^n-t)
%\Big)
%$$
%$$
% = \frac {1}{(t^2-t)(t^2-1)}\Big(
%t^{2n-1}-t^{2k+n-1}-t^{n+1} + t^{n+2k}+t-t^n
%\Big)
%$$
%$$
% = \frac {1}{(t-1)(t^2-1)}\Big(
%t^{2n-2}-t^{2k+n-2}-t^{n} + t^{n+2k-1}+1-t^{n-1}
%\Big)
%$$
%$$
% = \frac {1}{(t-1)(t^2-1)}\Big(
%(t^{2k}-1)(t-1)t^{n-2} - t^{2k+n-1}-t^{n-2}+t^{n-1}+t^{2k+n-2}+
%t^{2n-2}-t^{2k+n-2}-t^{n} + t^{n+2k-1}+1-t^{n-1}
%\Big)
%$$
%$$
% = \frac {1}{(t-1)(t^2-1)}\Big(
%(t^{2k}-1)(t-1)t^{n-2} -t^{n-2}+
%t^{2n-2}-t^{n} +1
%\Big)
%$$
%$$=
\Big( \frac {(uv)^{2k}-1}{(uv)^2-1}\Big)(uv)^{n-2} + \frac {
((uv)^{n-2}-1)((uv)^n-1)}{(uv-1)^2(uv+1)}.
$$
Then we have 
$$
E(H;u,v) = \sum_{k\geq 0}\Big( \frac {(uv)^{2k}-1}{(uv)^2-1}\Big)(uv)^{n-2} E(P_k;u,v) 
+
\frac {((uv)^{n-2}-1)((uv)^n-1)^2}{(uv-1)^3(uv+1)}
$$
and
$$
(uv)E(X_W;u,v) = \sum_{k\geq 0}\Big( \frac {(uv)^{2k}-1}{(uv)^2-1}\Big) E(P_k;u,v) 
 - \frac{(uv)^n-1}{(uv)^2-1}.
$$
\end{proof}

Thus, it is natural to expect that 
\begin{equation}\label{maybe}
(uv)E(X_W;u,v) = E_{st}(Y_W;u,v)
 - \frac{(uv)^n-1}{(uv)^2-1}
\end{equation}
which is equivalent to the statement that 
the stringy Hodge numbers $h^{p,q}_{st}(Y_W)$ are well defined and there holds
$$
h^{p,q}_{st}(Y_W)=\left\{
\begin{array}{ll}
h^{p-1,q-1}(X_W)+1,&p=q{\rm ~and~}p{~\rm is~even}\\
h^{p-1,q-1}(X_W),&{\rm else}.
\end{array}
\right.
$$
The formula \eqref{maybe} would follow from Proposition \ref{almost1} as long as the local contribution of 
the singularity of $Pf$ along $P_k$ is given by
$$
\frac {(uv)^{2k}-1}{(uv)^2-1}.
$$
This certainly holds for the nonsingular points, i.e. $k=1$. However, it \emph{fails} for the $k=2$ locus!
Specifically, we have the following.
\begin{proposition}
The singularities of $Y_W$ along the locus $P_2$ have a Zariski locally trivial resolution with the fibers isomorphic to $G(2,4)$.
The discrepancy is $3$. The local contribution of the singularity is given by 
$$
\frac {(uv)^2+uv+1}{uv+1}
$$
and \emph{is not} a polynomial.
\end{proposition}

\begin{proof}
Since the calculation depends only on the singularity, we may as well consider the locus of 
forms of rank $2$ in $\PP\Lambda^2 V$ for $\dim V=6$. When we blow it up, we get a 
resolution of singularities $\pi:\widehat{Pf}\to Pf(V)$ isomorphic to the bundle $\PP \Lambda^2 Q_2^\dual$ 
over the Grassmannian $G(2,V)$. The map to $Pf(V)$ is given by interpreting a form on $Q_2$ as 
a form on $V$. The exceptional divisor $D$ is the locus of forms of rank $2$ on $Q_2$.

\smallskip
We have 
$$K_{\widehat{Pf}}=\pi^*K_{Pf} + \alpha D = -12 \pi^*\xi + \alpha D
$$
where we use the fact that $Pf(V)$ is a hypersurface of degree $3$ in $\PP^{14}$ and denote
by $\xi$ the hyperplane class.
For the map $\mu: \widehat{Pf}\simeq \PP \Lambda^2 Q_2^\dual\to G(2,V)$  we have
$$
K_{\widehat{Pf}} = -6c_1(\mathcal O(1)) + \mu^* c_1(\Lambda^2 Q_2).
$$
We now observe that $c_1(\mathcal O(1))= \pi^*\xi$ to get
$$
\alpha D = 6 c_1(\mathcal O(1))+\mu^* c_1(\Lambda^2 Q_2).
$$
The divisor $D$ is a degree two hypersurface in the fibers of $\mu$, namely a $G(2,4)\subset\PP^5$.
Thus, it intersects the line in the fiber by $2$, and we get
$2\alpha = 6$ thus $\alpha=3$. 

\smallskip
The rest of the statement of the Proposition follows from Definition \ref{s} of the local contribution
and the standard formula $E(G(2,4);u,v)=(uv)^4+(uv)^3+2(uv)^2+uv+1$.
\end{proof}

\begin{remark}
As a consequence of the above calculation, the equation \eqref{maybe} fails for $n=8$, since the contribution
of $P_2$ is different from the one that's needed, and $P_{\geq 3}$ are empty.
\end{remark}

\begin{remark}
The desired contribution $\frac{(uv)^{2k}-1}{(uv)^2-1}$ of the singularity of $P_k$ would be achieved,
if the definition of the stringy Hodge numbers were given with different discrepancies. For example,
in the $k=2$ case the discrepancy $2$ would result in the local contribution
$$
E(G(2,4);u,v)\Big(\frac{uv-1}{(uv)^3-1}\Big) = (uv)^2+1
$$
as desired.
More generally, one can construct a resolution of $Pf(V)$ as a proper preimage of the Pfaffian divisor
in the space of complete skew forms on $V$. The corresponding exceptional divisors $D_k$ correspond 
to forms of corank $2k$. We would have the desired equality of the $E$-functions if the discrepancies of 
$D_k$ were $2k^2-3k$. However, a calculation similar to the one in Section \ref{sec.log} shows that the 
discrepancies are $2k^2-2k-1$. Thus it appears that in order to understand generalized double mirror phenomenon for varieties of different dimensions, one needs to adjust the definition 
of stringy Hodge numbers! We hope to explore this observation in the future, but at this moment, it remains a
puzzling phenomenon.
\end{remark}

\section{General Pfaffian double mirror correspondence}\label{sec7}
In this section we adapt the techniques of Sections \ref{sec4} and \ref{proof}  to the case of double mirror complete intersections
in general Pfaffian varieties of forms on the spaces of odd dimension.

\begin{definition} Let $V$ be  a complex vector space of odd dimension $n\geq 5$. 
For any $k\in \{1,\ldots, \frac 12(n-1)\}$ we define the variety $Pf(2k,V)$ to be the subvariety of the 
space $\PP\Lambda^2V^\dual$ of nontrivial skew forms on $V$ up to scaling, defined by the condition
that the rank of the form does not exceed $2k$. We define $Pf^\circ(2k,V)$ to be the locally closed subset
of forms of rank exactly $2k$. If the dependence on $V$ is not important, we will simply use $Pf(2k,n)$ 
and $Pf^\circ(2k,n)$.
\end{definition}

\begin{definition}\label{defgen}
Assume  that $k\neq \frac 12(n-1)$.
Let $W\subset \Lambda^2 V^\dual$ be a generic subspace of dimension $nk$. Then we can define two varieties
$X_W$ and $Y_W$ as follows. 
\begin{itemize}
\item
$X_W$ is the set of forms $y\in Pf(2k,V^\dual)$ such that $\langle y, w\rangle=0$ for all $w\in W$. Here 
we use the natural nondegenerate pairing between $\Lambda^2 V$ and $\Lambda^2 V^\dual$.
\item 
$Y_W$ is the complete intersection of $\PP W$ and $Pf(n-1-2k, V)$ in $\PP(\Lambda^2V^\dual)$.
\end{itemize}
\end{definition}

\begin{remark}\label{switch}
Definition \ref{defgen} contains the definition of Section \ref{sec2} as a special case $k=1$. We can also observe
that interchange of $(V,k,W)$ with $(V^\dual, \frac 12(n-1)-k,Ann(W))$ switches $X_W$ and $Y_W$. 
\end{remark}

The following result is well known to the experts but we could not find a suitable reference.
\begin{proposition}\label{kpf}
For positive integers $k,n$ with $2k\leq n$, the variety $Pf(2k,V^\dual)$ is Gorenstein, with anticanonical class given by $kn \xi$ where 
$\xi$ is the pullback of the hyperplane class of $\PP\Lambda^2V$.
\end{proposition}

\begin{proof}
Consider the non-log resolution $\pi:\PP \Lambda^2 Q^\dual \to Pf(2k,V^\dual)$
given by the space of pairs $(V_1,w)$ where $V_1$ is an $n-2k$ dimensional
subspace of $V^\dual$ in the kernel of $w\in Pf(2k,V^\dual)$. Here $\PP \Lambda^2 Q^\dual$ is the projective bundle
over the Grassmannian $G(n-2k,V^\dual)$. We denote this map by $\mu$.

\smallskip
By a standard calculation, we get
$$
K_{\PP \Lambda^2 Q^\dual} = -(n+1-2k)\mu^*c_1(Q) - k(2k-1)\xi
$$
where $\xi=c_1({\mathcal O}(1))$. The exceptional divisor $E$ of $\pi$
is the locus of degenerate forms in $\PP \Lambda^2 Q^\dual$, which is
the locus where the $k$-th power of the natural map $\mu^* \Lambda^2 Q \to
{\mathcal O}(1)$ is zero. This $k$-th power is a map
$$
\mu^*\Lambda^{2k} Q \to {\mathcal O}(k)
$$
which means that $c_1(E)=k\xi - \mu^*c_1(Q)$.

\smallskip
Since $Pf(2k,V^\dual)$ is Gorenstein (see \cite{JPW}) we have
$$
\pi^* K_{Pf(2k,V)} = K_{\PP \Lambda^2 Q^\dual} - \alpha c_1(E)
$$
for some $\alpha$. In view of the above calculations, this
is equivalent
to
$$
\pi^*K_{Pf(2k,V)} +nk\xi = (n+1-2k-\alpha)c_1(E).
$$
Since $E$ is the exceptional divisor, and the left hand side is a pullback from $Pf(2k,V^\dual)$, both sides are zero.
\end{proof}

\begin{corollary}
The varieties $X_W$ and $Y_W$ are Gorenstein with trivial canonical class of dimension
$nk-2k^2-k-1$.
\end{corollary}

\begin{proof}
We have already observed this in the case $k=1$ and $k=\frac 12(n-3)$.
In general,
Proposition \ref{kpf} implies the statement for $X_W$ by the adjunction formula. The statement for $Y_W$ then follows from Remark \ref{switch}.
\end{proof}

We would like to state the following meta-mathematical conjecture for a vector space $V$ of odd dimension $n$, a number 
$k\in \{1,\ldots, \frac 12(n-3)\}$ and a generic subspace $W\subset \Lambda^2 V^\dual$ of dimension $nk$:

\medskip
{\begin{center}\emph{
The varieties $X_W$ and $Y_W$ are double mirror to each other.}
\end{center}}

\smallskip
\begin{remark}
The above conjecture can be viewed as a natural generalization of R\o dland's work. Note also that 
Kuznetsov has stated the analogous conjecture for the derived categories \cite[Conjecture 4.9]{Kuznetsov2}.
The best evidence in favor of this conjecture is provided by the Theorem \ref{main.main} below, which is one of the 
expected consequences of double mirror property. It generalizes our Theorem \ref{main}.
\end{remark}

\begin{theorem}\label{main.main}
The varieties $X_W$ and $Y_W$ have well-defined stringy Hodge numbers. Moreover, there holds 
$$
h_{st}^{p,q}(X_W)=h_{st}^{p.q}(Y_W).
$$
\end{theorem}

We delay the proof of Theorem \ref{main.main} until later. First, we need to extend the results of Sections 
\ref{sec2}, \ref{sec.log} and \ref{sec4} to this more general setting. We begin with the discussion of log resolutions of Pfaffian varieties 
$Pf(2k,n)$. 

\begin{definition}
Let $2k$ and $n$ be positive integers with $2k\leq n$, where we no longer assume that $n$ is odd. Let $V$ be 
a vector space of dimension $n$. We define the space $\widehat {Pf(2k,V)}$ of 
\emph{complete skew forms of rank $\leq 2k$} as the proper preimage of 
$Pf(2k,V)$ under the consecutive blowups in $\PP \Lambda^2 V^\dual$ of $Pf(2,V)$, proper preimage of $Pf(4,V)$, and so on, 
up to $Pf(2k-2,V)$.
\end{definition}

\begin{proposition}
The space $\widehat {Pf(2k,V)}$ is smooth. Its points are in one-to-one correspondence with flags
$$
0\subseteq F^0 \subset \cdots \subset F^l = V
$$
with $F^0$ of dimension $n-2k$
together with \emph{nondegenerate} forms $\CC w_i \in \PP \Lambda^2(F^{i+1}/F^i)^\dual$.
The map $\widehat{Pf(2k,V)}\to Pf(2k,V)\subset \PP\Lambda^2 V^\dual$ is given by 
interpreting a skew form on $F^{l}/F^{l-1}$ as a skew form on $F^l=V$. 
\end{proposition}

\begin{proof}
Follows from \cite{Bertram}.
\end{proof}

\begin{proposition}\label{discpf.all}
Let $V$ be a space of odd dimension $n$ and let $k<\frac 12(n-1)$ be a positive integer.
The map $\widehat{Pf(2k,V)}\to Pf(2k,V)$ gives a log resolution of $Pf(2k,V)$. The exceptional divisors $D_j$ correspond
to loci of complete forms with a subspace of dimension $2j-1$ present in the flag. The index $j$ satisfies
$\frac 12(n+3-2k)\leq j\leq \frac 12(n-1)$. The discrepancy of the divisor $D_j$ is given by 
$$
\alpha_j = \frac 12(2j+2k-n-1)(2j-1)-1.
$$
\end{proposition}

\begin{proof}
The proof is completely analogous to that of Theorem \ref{discpf}. We consider the projective bundle $Z\to \widehat{Pf(2k,V)}$
by looking at $F^{-1}\subset F^0$ of codimension one, i.e. 
$$Z= \widehat{Pf(2k,V)}\times_{G(r,n)}{Fl(r-1,r,n)}.$$
where we use $r=n-2k$ to denote the dimension of $F^0$.
We have the commutative diagram
$$
\begin{array}{ccccccc}
&&\widehat{Pf(2k,V)}& \longleftarrow & Z&&\\
&&\downarrow &                 & \downarrow&\searrow&    \\
Pf(2k,V)&\longleftarrow&\PP \Lambda^2 Q_{r}^\dual&\longleftarrow&\PP \Lambda^2 Q_{r,Fl}^\dual&\longrightarrow&\PP \Lambda^2 Q_{r-1}^\dual \\
&&\downarrow &                 & \downarrow& & \downarrow   \\
&&G(r,V)&\longleftarrow&Fl(r-1,r,V)&\longrightarrow&G(r-1,V)
\end{array}
$$
The preimages of $D_j$ in $Z$ are denoted by $E_j$. We have the following three equalities in the Picard group of $Z$ (where we drop
the pullbacks from the notation to simplify it). We also use $r=n-2k$ and denote by $\xi$ the hyperplane class on $\PP\Lambda^2V^\dual$.
\begin{equation}\label{eq1}
K_Z = K_{Pf(2k,V)}+ \sum_j \alpha_j E_j  + K_{Fl} -K_{G(r,n)}
\end{equation}
\begin{equation}\label{eq2}
K_Z=K_{Fl}-k(2k-1)\xi+(2k-1)c_1(Q_{r})+\sum_j (\frac 12(2j-r-1)(2j-r-2)-1)E_j
\end{equation}
\begin{equation}\label{eq3}
K_Z=K_{G(r-1,n)}-k(2k+1)\xi +2k c_1(Q_{r-1}) + \sum_j (\frac 12(2j-r)(2j-r-1)-1)E_j
\end{equation}
We take a linear combination of the equations \eqref{eq1}, \eqref{eq2} and \eqref{eq3} with coefficients 
$2$, $(r-1)$ and $(-r-1)$ respectively. This implies the desired  equality 
$$\alpha_j=\frac 12(2j-r-1)(2j-1)-1$$ 
provided that 
$$
0=2K_{Fl}-2K_{G(r,n)}+(r-1)K_{Fl}+(r-1)(2k-1)c_1(Q_r)-(r+1)K_{G(r-1,n)}-2k(r+1)c_1(Q_{r-1})
$$
which is equivalent to a formula for the canonical class of the partial flag variety $Fl(r-1,r,n)$ 
$$
K_{Fl} =-rc_1(Q_{r-1})-(2k+1)c_1(Q_r)
$$
which is proved analogously to the $r=3$ case of Theorem \ref{discpf}.
\end{proof}

We now proceed to generalize the results of Section \ref{sec4}. Recall that $e_{2i}=e_{2i}(u,v)$ is the $E$-function of the space of nondegenerate forms on a $2i$-dimensional space, up to scaling, and $g_{k,n}$ is the $E$-function of the Grassmannian $G(k,n)$.
Our main technical result is the following proposition.

\begin{proposition}\label{technical}
For an odd integer $n$ and any $1\leq k\leq \frac 12(n-3)$ there holds
$$
\sum_{1\leq i\leq \frac 12(n-1) }e_{2i}g_{n-2i,n} \prod_{j=k-i+1}^{\frac 12(n-1)-i} 
%\left(
\frac{(uv)^{2j}-1}{(uv)^{2j-2k+2i}-1}
%\right)
=
%\left(
\frac{(uv)^{nk}-1}{uv-1}
%\right)
\prod_{j=k+1}^{\frac 12(n-1)}
%\left(
\frac{(uv)^{2j}-1}{(uv)^{2j-2k}-1}
%\right)
.
$$
\end{proposition}

\begin{proof}
We will prove this result by induction on $n+k$. The base case is easy to establish by direct calculation, which we leave to the reader.

\smallskip
Suppose that the result of the proposition is proved for all smaller values of $n+k$. Specifically, we will use these formulas
for $(k-1,n)$ and $(k,n-2)$. We need to be careful in the situation when these pairs are not in the acceptable range of $(k,n)$.
This happens when $k=1$ or  $k=\frac 12(n-3)$. 
Note that only $i$ in the range $1\leq i\leq k$ contribute to the left hand side, so $k=1$ case follows from the formula
$$g_{n-2,n} = \frac{((uv)^n-1)((uv)^{n-1}-1)}{((uv)^2-1)((uv)-1)}.$$ 
In the $k=\frac 12(n-3)$ case the product on the  left hand side only has $j=\frac 12(n-1)-i$ and we get
$$
\sum_{1\leq i\leq \frac 12(n-1) }e_{2i}g_{n-2i,n}
% \left(
\frac{(uv)^{n-2i-1}-1}{(uv)^{2}-1}
%\right)
$$
which allows us to use Proposition \ref{sum} to settle this case.

\smallskip
In the general case $1<k<\frac 12(n-3)$ we may now assume the equations 
\begin{align*}
\sum_{1\leq i\leq \frac 12(n-1) }e_{2i}g_{n-2i,n} \prod_{j=k-i}^{\frac 12(n-1)-i} 
\frac{(uv)^{2j}-1}{(uv)^{2j-2k+2+2i}-1}
\\
=\frac{(uv)^{n(k-1)}-1}{uv-1}
\prod_{j=k}^{\frac 12(n-1)}
\frac{(uv)^{2j}-1}{(uv)^{2j-2k+2}-1}
\end{align*}
and
\begin{align*}
\sum_{1\leq i\leq \frac 12(n-3) }e_{2i}g_{n-2i-2,n-2} \prod_{j=k-i+1}^{\frac 12(n-3)-i} 
\frac{(uv)^{2j}-1}{(uv)^{2j-2k+2i}-1}
\\
=
\frac{(uv)^{(n-2)k}-1}{uv-1}
\prod_{j=k+1}^{\frac 12(n-1)}
\frac{(uv)^{2j}-1}{(uv)^{2j-2k}-1}
.
\end{align*}
We use the relation $g_{n-2i-2,n-2}=g_{n-2i,n}\frac{((uv)^{n-2i}-1)((uv)^{n-2i-1}-1)}{((uv)^n-1)((uv)^{n-1}-1)}$
to rewrite the above two equations as 
\begin{equation}\label{t1}
\begin{split}
\left(
\frac {(uv)^{2k-2i}-1}{(uv)^{n+1-2k}-1}
\right)
\sum_{1\leq i\leq \frac 12(n-1) }e_{2i}g_{n-2i,n} \prod_{j=k-i+1}^{\frac 12(n-1)-i} 
\frac{(uv)^{2j}-1}{(uv)^{2j-2k+2i}-1}\newline
\\
=\frac {(uv)^{nk-n}-1}{uv-1}\prod_{j=k}^{\frac 12(n-1)} \frac{(uv)^{2j}-1}{(uv)^{2j-2k+2}-1}
\end{split}
\end{equation}
and
\begin{equation}\label{t2}
\begin{split}
\left(
\frac{((uv)^{n-2i}-1)((uv)^{n-1-2k}-1)}{((uv)^n-1)((uv)^{n-1}-1)}
\right)
\sum_{1\leq i\leq \frac 12(n-1) }e_{2i}g_{n-2i,n} \prod_{j=k-i+1}^{\frac 12(n-1)-i} 
\frac{(uv)^{2j}-1}{(uv)^{2j-2k+2i}-1}\newline
\\
=\frac {(uv)^{nk-2k}-1}{uv-1}\prod_{j=k+1}^{\frac 12(n-3)} \frac{(uv)^{2j}-1}{(uv)^{2j-2k}-1}.
\end{split}
\end{equation}
The desired equality then follows from taking the linear combination of the equations \eqref{t1} and \eqref{t2} 
with coefficients
$$(uv)^n((uv)^{n+1-2k}-1),~~-\frac {(uv)^{2k}((uv)^n-1)((uv)^{n-1}-1)}{((uv)^{n-1-2k}-1)}.$$
Details are left to the reader.
\end{proof}

We are now able to prove a formula for the stringy $E$-functions of the Pfaffian varieties $Pf(2k,n)$.
\begin{theorem}\label{pfst2k}
For a vector space $V$ of  odd dimension $n$ and 
any $1\leq k\leq \frac 12 (n-1)$ there holds
$$
E_{st}(Pf(2k,V)) = \frac{(uv)^{nk}-1}{uv-1}
\prod_{j=k+1}^{\frac 12(n-1)} \frac {(uv)^{2j}-1}{(uv)^{2j-2k}-1}.
$$
\end{theorem}

\begin{proof}
We remark that the case $k=1$ is straightforward, since $Pf(2k,V)=G(2,V^\dual)$.
Otherwise, for a fixed $n-2k$ we perform induction on $n$ with the base case described above.

\smallskip
To prove the induction step, consider the log resolution of singularities 
$\widehat{Pf(2k,V)}\to Pf(2k,V)$. For each $i<k$ the strata of  $\widehat{Pf(2k,V)}$ over the space $Pf^\circ(2i,V)$ of forms of 
rank exactly $2i$ form Zariski locally trivial fibrations. The preimage of a point $\CC w\in Pf^\circ(2i,V)$ is isomorphic to the log resolution of $Pf(2k-2i, Ker(w))$. The contribution of $Pf^\circ(2i,V)$ is then seen to equal 
\begin{equation}\label{contrib.all}
E(Pf^\circ(2i,V)) E_{st}(Pf(2k-2i, Ker(w))) \frac{uv-1}{(uv)^{\alpha_{j}+1}-1}
\end{equation}
where $\alpha_j$ is the discrepancy calculated in Proposition \ref{discpf.all} to be 
$$
\frac 12(2j+2k-n-1)(2j-1)-1 = (k-i)(n-2i)
$$
since $2j-1= n-2i$. By induction assumption, the contribution \eqref{contrib.all} is equal to 
$$
E(Pf^\circ(2i,V)) \prod_{j=k-i+1}^{\frac 12(n-1)-i} \frac {(uv)^{2j}-1}{(uv)^{2j-2k+2i}-1}.
$$

\smallskip
Thus we have 
\begin{align*}
E_{st}(Pf(2k,V)) =& E(Pf^\circ(2k,V))+\sum_{1\leq i<k}E(Pf^\circ(2i,V)) \prod_{j=k-i+1}^{\frac 12(n-1)-i} \frac {(uv)^{2j}-1}{(uv)^{2j-2k+2i}-1}\\
=& \sum_{1\leq i\leq k}E(Pf^\circ(2i,V)) \prod_{j=k-i+1}^{\frac 12(n-1)-i} \frac {(uv)^{2j}-1}{(uv)^{2j-2k+2i}-1}
\\
=&\sum_{1\leq i\leq k}e_{2i}g_{n-2i,n} \prod_{j=k-i+1}^{\frac 12(n-1)-i} \frac {(uv)^{2j}-1}{(uv)^{2j-2k+2i}-1}.
\end{align*}
We can change the index of summation to $1\leq i\leq \frac 12(n-1)$ since the subsequent terms will have a zero factor in the product.
Then Proposition \ref{technical} finishes the proof.
\end{proof}

\begin{remark}
It is worth mentioning that $E_{st}(Pf(2k,V))$ is a polynomial in $uv$. Indeed, it is a product of the $E$-function of a projective space 
with a Gaussian binomial coefficient in $(uv)^2$. See Section \ref{Appendix} for more information on the latter.
\end{remark}

\section{Proof of the equality of $E$-functions: general case}\label{sec8}
In this section we prove our main result Theorem \ref{main.main} by using some of the more technical statements collected in the Appendix
(Section \ref{Appendix}).

\medskip
Recall that we are given an odd integer $n\geq 5$, a positive integer $k<\frac 12(n-1)$, a vector space $V$ of dimension $n$ 
and a generic subspace of $\Lambda^2 V^\dual$ of dimension $nk$. This allows us to define two varieties $X_W$ and $Y_W$ 
which are Calabi-Yau complete intersections in $Pf(2k,V^\dual)$ and $Pf(2k,V)$ respectively.

\begin{proposition}
Stringy Hodge numbers of $X_W$ and $Y_W$ are well-defined.
\end{proposition}

\begin{proof}
Consider the log resolution $\widehat{Pf(2k,V^\dual)}\to Pf(2k,V^\dual)$ given by the complete skew forms on $V^\dual$ 
of rank at most $2k$. The restriction of this resolution to the preimage of $X_W$ is a 
complete intersection of $\widehat{Pf(2k,V^\dual)}$ with $\PP Ann(W)$. It gives a log resolution of
$X_W$, since $W$ is generic. Moreover, $\widehat{Pf(2k,V^\dual)}\to Pf(2k,V^\dual)$ and therefore its restriction to the preimage 
of $X_W$ are Zariski locally trivial. It remains to observe that the contributions of singular points along $Pf(2p,V^\dual)\cap X_W$
in the sense of Definition \ref{s} are calculated in the proof of Theorem \ref{pfst2k} as
\begin{equation}\label{ss}
S(p,k,n;u,v)=\prod_{j=k+1-p}^{\frac 12(n-1)-p} \frac{(uv)^{2j}-1}{(uv)^{2j-2k+2p}-1}
\end{equation}
and are polynomials in $uv$. Thus the stringy $E$-function of $X_W$ is a linear combination of products of polynomials and is a polynomial.
The statement also applies to $Y_W$ by Remark \ref{switch}.
\end{proof}

We are now ready to prove our main result, up to a technical statement on weighted $E$-functions of hyperplane cuts of 
the Pfaffian varieties, which is relegated to Section \ref{Appendix}.
\begin{proof}(of Theorem \ref{main.main})
Consider the Cayley hypersurfaces $H\subset Pf(2k,V^\dual)\times \PP W$ of the forms $w\in \Lambda^2V^\dual$ of rank $\leq 2k$ 
and forms $\alpha\in \PP W$ with the property $\langle w,\alpha \rangle =0$ where we use the natural pairing between
$\Lambda^2 V$ and $\Lambda^2 V^\dual$. The projection of $H\to Pf(2k,V^\dual)$ can be viewed as disjoint union of two Zariski locally trivial fibrations over $X_W$ and its complement. The fibration over $X_W$ has fibers $\PP W=\PP^{nk-1}$ whereas the fibration over 
the complement of $X_W$ has fibers $\PP^{nk-2}$. Note that the same statements apply to the preimage $\widehat H$ of $H$ in 
$\widehat{ Pf(2k,V^\dual)}\times \PP W$. 

\smallskip
We consider the stratification of $H$ into loci of different rank 
$$
H=\bigsqcup_{1\leq p\leq k} H_p
$$
where $H_p$ has forms $w$ of rank $2p$. We then consider the weighted sum of the $E$-polynomials of $H_p$
\begin{equation}\label{weighted}
\tilde E(H;u,v) = \sum_{p=1}^k S(p,k,n;u,v)E(H_p;u,v)
\end{equation}
with $S(p,k,n;u,v)$ the local contribution as in \eqref{ss}. The usual arguments now imply
$$
\tilde E(H;u,v) = E_{st}(Pf(2k,V)) \frac{(uv)^{nk-1}-1}{uv-1} + E_{st}(X_W)(uv)^{nk-1}
$$
and 
\begin{equation}\label{Hside}
\tilde E(H;u,v)=
\frac{((uv)^{nk-1}-1)(uv)^{nk}-1)}{(uv-1)^2} \prod_{j=k+1}^{\frac 12(n-1)}
\frac{(uv)^{2j-1}-1}{(uv)^{2j-2k}-1} 
+ E_{st}(X_W)(uv)^{nk-1}
\end{equation}
by Theorem \ref{pfst2k}.

\smallskip
As in Section \ref{proof}, we now consider the projection of $H$ to $\PP W$. For each $i$ there is a locus $Y_{W,i}$ of forms 
of rank $2i$ in $\PP W$. We have 
$$Y_W = \overline{Y_{W,n-1-2k}}= \bigsqcup_{1\leq i\leq \frac 12(n-1)-k} Y_{W,i}.$$
We observe that the fibers of the projection of $H$ onto $Y_{W_i}$ are isomorphic to the loci of skew 
forms of rank $\leq 2k$ on $V^\dual$
which are orthogonal to a specific skew form of rank $2i$ on $V$. While we do not claim that these fibrations are Zariski locally trivial, we nonetheless observe that $E$-function is multiplicative on them. Moreover, the same applies to the restrictions of the fibrations
to $H_p\subseteq H$.  Indeed, by passing to the symplectic frame bundles as in Section \ref{proof} we get a disjoint union of iterates of Zariski locally trivial fibrations. Specifically, the statement holds for the symplectic frame bundle over the universal Cayley hypersurface in $Pf^\circ(2p,V^\dual)\times Pf^\circ (2i,V)$ of $(v_1,\ldots, v_{2p})\in V^{2p}$ with $\CC w =\CC (v_1\wedge v_2 +\ldots+w_{2p-1}\wedge v_{2p})$. We add elements $v_i$ one-by one and separate the loci of different dimensions of $Ker(\alpha)\cap Span(v_1,\ldots, v_{p_1})$,  different orthogonality conditions between $v_i$-s and spans of $v_{<i}$, and whether $\sum_{i=1}^{p_1} \alpha(v_{2i-1}\wedge v_{2i}) =0$ or $\neq 0$ for various $p_1\leq p$.

\smallskip
Thus we have 
\begin{equation}\label{Yside}
\tilde E(H;u,v)
=\sum_{i=1}^{\frac 12(n-1)}
\tilde E(Pf(2k,n)\cap \langle\cdot ,\alpha_i\rangle=0)E(Y_{W,i})
\end{equation}
where $\tilde E(Pf(2k,n)\cap \langle\cdot ,\alpha_i\rangle=0)$ is defined as in \eqref{weighted} as the sum of $E$-functions 
of the loci of rank $2p$ in the hypersurface in $Pf(2k,n)$ cut out by a form $\alpha_i$ of rank $2i$.
We now use the result of Proposition \ref{app.main} that $E(Pf(2k,n)\cap \langle\cdot ,\alpha_i\rangle=0)$
is given by 
$$\frac {(uv)^{nk-1}-1}{uv-1}\prod_{j=k+1}^{\frac 12(n-1)} \frac {(uv)^{2j}-1}{(uv)^{2j-2k}-1}
+(uv)^{nk-1}S(i,\frac 12(n-1)-k,n)
$$
where 
$$S(i,\frac 12(n-1)-k,n)=\prod_{j=\frac 12(n-1)-k-i+1}^{\frac 12(n-1)-i} \frac{(uv)^{2j}-1}{(uv)^{2j-n+1+2k+2i}-1}
$$
is the contribution of the locus $Y_{W,i}$ to $E_{st}(Y_W)$. Observe that $S(i,\frac 12(n-1)-k,n)$ is zero for $i>\frac 12(n-1)-k$, 
since there is a  term $((uv)^0-1)$  in the product. Thus, equation \eqref{Yside} implies
\begin{align*}
\tilde E(H;u,v) = &
\frac {(uv)^{nk-1}-1}{uv-1}\prod_{j=k+1}^{\frac 12(n-1)} \frac {(uv)^{2j}-1}{(uv)^{2j-2k}-1}\sum_{i=1}^{\frac 12(n-1)}E(Y_{W,i})
\\
&
+(uv)^{nk-1}\sum_{i=1}^{\frac 12(n-1)-k}
S(i,\frac 12(n-1)-k,n)E(Y_{W,i})
\\
=
&
E(\PP W)\frac {(uv)^{nk-1}-1}{uv-1}\prod_{j=k+1}^{\frac 12(n-1)} \frac {(uv)^{2j}-1}{(uv)^{2j-2k}-1}
+(uv)^{nk-1}E_{st}(Y_W).
\end{align*}
Together with the equation \eqref{Hside}, this finishes the proof of Theorem \ref{main.main}.
\end{proof}

\section{Comments.}\label{sec.last}
In this section we collect several open questions raised by the construction of this paper that we hope to address in
the future.

\smallskip
\begin{remark}
Theorems \ref{main} and \ref{main.main} show the equality of the (stringy) Hodge numbers of $X_{W}$ and $Y_{W}$ but do not provide 
explicit formulas for them. In the Grassmannian case the Lefschetz hyperplane theorem on $X_W$ side implies that these numbers $h^{p,q}$ 
are zero unless $p=q$
or $p+q=n-4$, with the former coming from the restriction of the cohomology of $G(2,V)$. As a result, Hodge numbers $h^{p,q}$  are computable by a Riemann-Roch calculation for the exterior powers of cotangent bundle on $X_W$. However, we are not aware of an explicit formula. In the general Pfaffian setting, 
we do not even have a ready algorithm for computing  stringy Hodge numbers of $X_W$ and $Y_W$.
\end{remark}

\smallskip
\begin{remark}
It would be very interesting to lift the equality of numbers of Theorem \ref{main} to a statement about vector spaces. Heuristically, one 
expects a family of spaces, with a connection, that interpolates from the somehow defined stringy cohomology of $X_W$ to that of $Y_W$. 
The first step in constructing such family  is likely a construction of stringy cohomology vector space(s)
of the Pfaffian variety as indicated in Remark \ref{remspace}.
\end{remark}

\smallskip
\begin{remark}
It is natural to conjecture that for odd $n$ the appropriately defined elliptic genera of $X_W$ and $Y_W$ coincide.
We conjecture that $Ell(X_W;y,q)=Ell(Y_W;y,q)$ where elliptic genus of the singular varieties is defined in \cite{ellgen}.
Similarly, we would like to have a vector space version of this identity, which would amount to a construction of the family of vertex algebras that interpolates between the cohomology of the chiral de Rham complex  of $X_W$ and $Y_W$. In the case of one or both of these spaces being singular, one would need 
to somehow extend the definition of the chiral de Rham complex. At present, such construction is not known even for $n=7$.
\end{remark}

\smallskip
\begin{remark}
It is natural to try and understand a higher dimensional analog of mirror construction of \cite{Rodland}. We would also want to relate it 
to the work of Batyrev, Ciocan-Fontanine, Kim and van Straten \cite{Batetall}. While it is straightforward to write a conjectural one-dimensional subfamily of $Y_W$ that generalizes the one from \cite{Rodland} to higher dimensions, we do not yet understand the structure of its singularities. This prevents us from calculating the stringy Hodge numbers of the cyclic quotient, which is required in order to establish the mirror duality of stringy Hodge numbers. 
\end{remark}

\smallskip
\begin{remark}
Double mirror phenomenon predicts the equivalence of appropriately defined derived categories of $X_W$ and $Y_W$. 
These should be strongly crepant categorical resolutions of singularities in the sense of \cite{Kuznetsov}.
\end{remark}

\smallskip
\begin{remark}
Most of the calculations of the paper, including those in Section \ref{Appendix}. have analogs in the case of even $n$. As explained in Section \ref{sec.even},
one would need to redefine the stringy Hodge numbers to give them a proper geometric meaning.
\end{remark}

\section{Appendix}\label{Appendix}
The main goal of this appendix is to prove the technical statement on the weighted $E$-function of the hyperplane cut of 
the Pfaffian variety $Pf(2k,n)$ by a Pl\"ucker hyperplane $\alpha$ of rank $2i$. Our argument uses various identities of 
basic hypergeometric functions, including Jain's identity \cite{Jain}.

\medskip
Throughout the section we will use the notation $q=uv$, since all of the $E$-functions in question will depend on $uv$ only.
 We will also use the $q$-Pochhammer symbols and $q$-binomial symbols (also known as Gaussian binomial coefficients)
$$
(a;q)_k = \prod_{j=0}^{k-1}(1-aq^k),~~
\left(
\begin{array}{c}
m \\ r
\end{array}
\right)_{q}
=\left\{
\begin{array}{ll}
\frac {\prod_{i=0}^{r-1} (1-q^{m-i})}
{\prod_{i=0}^{r-1} (1-q^{i+1})},& 0\leq r\leq m\\
0,&{\rm else}
\end{array}
\right.
$$
as well as the basic (also called $q$-) hypergeometric functions
$$
_{r}\phi_s
\left(
\begin{array}{c}
a_1,\ldots,a_r\\
b_1,\ldots,b_s
\end{array}
;q,z
\right)
=\sum_{n\geq 0}
\frac {(a_1;q)_n\cdots (a_r;q)_n}
{(q;q)_n(b_1;q)_n\cdots(b_s;q)_n}
((-1)^nq^{\frac12 n(n-1)})^{1+s-r}z^n.
$$
Our main reference is the book of Gasper and Rahman \cite{GR}.

\begin{remark}\label{tohyper}
The terms of the series $\sum_{n\geq 0}c_n$ used to define a basic hypergeometric function have the 
property that the ratio of consecutive terms $c_{n+1}/c_n$ is a rational function of $q^n$. 
Specifically, one has (see also \cite[equation 1.2.26]{GR})
$$
\frac {c_{n+1}}{c_n} = 
\frac
{(1-a_1q^n)\cdots(1-a_rq^n)}
{(1-q^{n+1})(1-b_1q^n)\cdots(1-b_sq^n)}
(-q^n)^{1+s-r}z.
$$
Vice versa,
any series with such a recursive relation can be written as the product of $c_0$ 
with a hypergeometric function that encodes the inverse roots of the aforementioned rational function.
\end{remark}

\begin{remark}
All of the basic hypergeometric functions in this paper are terminating, which means that the terms 
of the series are eventually zero. Thus, convergence is never an issue. 
\end{remark}

\medskip
We start the discussion of this section by introducing notation for $E$-polynomials of various spaces of interest.
\begin{definition}
We define the following polynomials, for the appropriate ranges of  the  indices.
\begin{itemize}
%\item 
%$e_{2i}(q)$ is  the $E$-polynomial of the space of nondegenerate skew forms in $\PP\Lambda^2\CC^{2i}$. 
\item
$gr_{k,n}(q)$ is the $E$-polynomial of the Grassmannian of dimension $k$ subspaces of $\CC^n$.
\item
$l_{k,i,n}$ is the $E$-polynomial of the variety of isotropic subspaces of dimension $2k$
for a form of rank $2i$ on a dimension $n$ space $V$.
\item
$f^\circ_{k,i,n}(q)$ is the $E$-polynomial of the intersection of the locus of skew forms of 
rank $2k$ on $V^\dual$
by the hyperplane $\alpha=0$ for a skew form $\alpha$ on $V$ of rank $2i$, for a vector space $V$ of odd dimension $n$.
\item
$f_{k,i,n}(q)$ is the weighted $E$-function of the 
 intersection of the locus of skew forms of 
rank $\leq 2k$ on $V^\dual$
by the hyperplane $\alpha=0$ for a skew form $\alpha$ on $V$ of rank $2i$.
Specifically,
\begin{equation}\label{fcirctof}
f_{k,i,n}= \sum_{1\leq p\leq k} f^\circ_{p,i,n} \prod_{j=k+1-p}^{\frac 12(n-1)-p}\frac{q^{2j}-1}{q^{2j-2p}-1}
=\sum_{1\leq p\leq k} f^\circ_{p,i,n} 
\left(
\begin{array}{c}
\frac 12(n-1)-p \\ k-p
\end{array}
\right)_{q^2}
\end{equation}
\end{itemize}
\end{definition}

We start by calculating some of these polynomials.
\begin{proposition}\label{g}
For $1\leq k\leq n$ there holds
$$
gr_{k,n} = \prod_{j=1}^k \frac {q^{n-j+1}-1}{q^{k-j+1}-1}=
\left(
\begin{array}{c}
n\\
k
\end{array}
\right)_q.
$$
\end{proposition}

\begin{proof}
Left to the reader.
\end{proof}

\begin{proposition}\label{l}
For all positive integers $k$, $i$, $n$ there holds
$$
l_{k,i,n} = \sum_{0\leq r\leq 2k}gr_{r,n-2i}q^{(2k-r)(n-2i-r)}
\frac{ \prod_{j=i+r+1-2k}^{i} (1-q^{2j})}{\prod_{j=1}^{2k-r} (1-q^{j})}.
$$
\end{proposition}

\begin{proof}
Let $w$ be the form of rank $2i$ on an $n$-dimensional space $V$. Let $N$ be the kernel of $w$. Isotropic spaces 
$V_1$ of dimension $2k$ for $w$ are first separated into a disjoint union of loci with $\dim(N\cap V_1) =r$ for $0\leq r\leq 2k$.
A choice of $V_2=N\cap V_1$ amounts to a Zariski locally trivial fibration with fiber $G(r, n-2i)$, so to prove the statement of the proposition we need to show 
that the space of isotropic subspaces $V_1$ that contain a fixed $V_2$ has $E$-polynomial given by
$$
q^{(2k-r)(n-2i-r)}
\frac{ \prod_{j=i+r+1-2k}^{i} (q^{2j}-1)}{\prod_{j=1}^{2k-r} (q^{j}-1)}.
$$
To every choice of $V_1$ we associate the corresponding space $V_3=V_1/V_2$ which is an isotropic subspace of dimension 
$2k-r$ for the nondegenerate form $w$ on $V/N$. This is a Zariski locally trivial fibration with fibers $\CC^{(2k-r)(n-2i-r)}$.
Indeed, for a given $V_3$ different lifts of $(2k-r)$ basis elements to $V$ can be (independently) changed by an element of $N/V_1$. 
Therefore, to prove the statement of the proposition, we need to show that the space of dimension $(k-2r)$ isotropic subspaces $V_3$
of  $V/N=\CC^{2i}$ equipped with a nondegenerate skew form has  $E$-polynomial 
$$
\frac{ \prod_{j=i+r+1-2k}^{i} (q^{2j}-1)}{\prod_{j=1}^{2k-r} (q^{j}-1)}.
$$
As usual, we consider the ordered bases of $V_3$. The first vector can be chosen arbitrarily. The second vector 
is perpendicular to the first one but is linearly independent from it, etc. This gives
$$
(q^{2i}-1)(q^{2i-1}-q)\cdots (q^{2i-(2k-r-1)}-q^{2k-r-1})=q^{\frac 12j(j-1)}\prod_{j=1}^{i+r-1-2k}(q^{2j}-1)
$$
which then needs to be divided by the $E$-polynomial of $GL(2k-r,\CC)$ to finish the proof.
\end{proof}

The calculation of $f_{k,i,n}$ is more complicated. We start by reversing the formula \eqref{fcirctof}.
\begin{lemma}\label{ff}
There holds 
$$
f_{k,i,n}^\circ = \sum_{1\leq j \leq k} f_{j,i,n}(-1)^{k-j}q^{(k-j)(k-j-1)}
\left(
\begin{array}{c}
\frac 12(n-1)-j \\ k-j
\end{array}
\right)_{q^2}.
$$
\end{lemma}

\begin{proof}
This follows from \eqref{fcirctof} and the standard summation formula for Gaussian binomial coefficients
known as $q$-binomial theorem.
Specifically, by  \eqref{fcirctof}   we have
\begin{align*}
& \sum_{1\leq j \leq k} f_{j,i,n}(-1)^{k-j}q^{(k-j)(k-j-1)}
\left(
\begin{array}{c}
\frac 12(n-1)-j \\ k-j
\end{array}
\right)_{q^2}\\
&
=
\sum_{1\leq j \leq k} 
\sum_{1\leq p\leq j}
 f^\circ_{p,i,n} 
  (-1)^{k-j}q^{(k-j)(k-j-1)}
\left(
\begin{array}{c}
\frac 12(n-1)-j \\ k-j
\end{array}
\right)_{q^2}
\left(
\begin{array}{c}
\frac 12(n-1)-p \\ j-p
\end{array}
\right)_{q^2}
\\
&
=
\sum_{1\leq p \leq k} 
 f^\circ_{p,i,n} 
\sum_{s=0}^{k-p}
  (-1)^{s}q^{s(s-1)}
\left(
\begin{array}{c}
\frac 12(n-1)-k+s\\ s
\end{array}
\right)_{q^2}
\left(
\begin{array}{c}
\frac 12(n-1)-p \\ k-p-s
\end{array}
\right)_{q^2}
\\
&
=
\sum_{1\leq p \leq k} 
 f^\circ_{p,i,n} 
 \left(
\begin{array}{c}
\frac 12(n-1)-p \\ k-p
\end{array}
\right)_{q^2}
\sum_{s=0}^{k-p}
  (-1)^{s}q^{s(s-1)}
\left(
\begin{array}{c}
k-p\\ s
\end{array}
\right)_{q^2}
\\
&
=
\sum_{1\leq p \leq k} 
 f^\circ_{p,i,n} 
 \left(
\begin{array}{c}
\frac 12(n-1)-p \\ k-p
\end{array}
\right)_{q^2}
(1;q^2)_{k-p} 
=f^\circ_{k,i,n}.
\end{align*}
At the end of the calculation we used \cite[Exercise 1.2(vi)]{GR} and $(1;q^2)_{k-p}=\delta_k^p$ for $k\geq p$.

\end{proof}

The following proposition contains the key geometric idea behind the calculation of this section.
\begin{proposition}\label{newrec}
For each triple of positive integers $(k,i,n)$ with $n$ odd and $k,i\leq \frac 12(n-1)$
there holds 
$$
\sum_{p=1}^k gr_{n-2k,n-2p}f^\circ_{p,i,n} = \frac {q^{2k^2-k-1}-1}{q-1}gr_{2k,n} +q^{2k^2-k-1} l_{k,i,n}.
$$
\end{proposition}

\begin{proof}
Let $\alpha$ be a skew form of rank $2i$ on $V$.
Consider the space of pairs $(w,V_1)$ where $w$ is a form of rank at most $2k$ on $V^\dual$ and
$V_1$ is a $(n-2k)$-dimensional subspace in the kernel of $w$. This is the projective bundle $\PP\Lambda^2 Q^\dual$ over $G(n-2k,V^\dual)$
of relative skew forms on the universal quotient bundle of $G(n-2k,V^\dual)$. This is
 a (non-log) resolution of $Pf(2k,V^\dual)$ with the map defined by forgetting the space $V_1$. Consider the hypersurface $H_\alpha\in\PP\Lambda^2 Q^\dual$  defined by 
$\langle w,\alpha\rangle=0$. We can calculate the $E$-function of $H_\alpha$ in two ways which will lead to the statement of the proposition.

\smallskip
On the one hand, consider the projection of $H_\alpha$ to the Pfaffian $Pf(2k,V^\dual)$. The fiber over the locus $Pf^\circ(2p,V^\dual)$ is given by 
$G(n-2k,n-2p)$. It is a Zariski locally trivial fibration over the locus of forms in $Pf^\circ(2p,V^\dual)$ that are orthogonal to $\alpha$.
Thus, we get 
$$
E(H_\alpha)=\sum_{p=1}^k gr_{n-2k,n-2p}f^\circ_{p,i,n}.
$$
On the other hand, consider the projection of $H_\alpha$ to the Grassmannian $G(n-2k,n)$. The fiber over a point $V_1$ is either the projective space 
$\PP^{2k^2-k-1}$ or a hyperplane in it, depending on whether or not $Ann(V_1)$ is an isotropic space for $\alpha$. Thus we get
$$
E(H_\alpha)=(gr_{n-2k,n}-l_{k,i,n})E(\PP^{2k^2-k-2}) + l_{k,i,n}E(\PP^{2k^2-k-1})
$$
which equals to the right hand side of the equation of the proposition.
\end{proof}

We are now able to exhibit a recursive formula for $f_{k,i,n}$.
\begin{proposition}\label{newcor}
For every triple of positive integers $(k,i,n)$ with odd $n$ and $k,i\leq \frac 12(n-1)$ there holds
$$
\sum_{j=1}^k f_{j,i,n} q^{2(k-j)^2-(k-j)}
\frac
{(1-q^{n+1-2k})}
{(1-q^{n+1-2j})}
\frac {(q^{n+3-4k+2j};q^2)_{2k-2j}}{(q;q)_{2k-2j}}
$$
$$
= \frac {q^{2k^2-k-1}-1}{q-1}gr_{2k,n} +q^{2k^2-k-1} l_{k,i,n}.
$$
This equation determines $f_{k,i,n}$ uniquely. 
\end{proposition}

\begin{proof}
We combine  Proposition \ref{newrec} and Lemma \ref{ff}
to get
$$
\sum_{j=1}^k f_{j,i,n}\left(\sum_{j\leq p\leq k}
 (-1)^{p-j}q^{(p-j)(p-j-1)}
\left(
\begin{array}{c}
n-2p \\ n-2k
\end{array}
\right)_{q}
\left(
\begin{array}{c}
\frac 12(n-1)-j \\ p-j
\end{array}
\right)_{q^2}
\right)
$$
$$
= \frac {q^{2k^2-k-1}-1}{q-1}gr_{2k,n} +q^{2k^2-k-1} l_{k,i,n}.
$$
Then the equation of the proposition follows from
\begin{equation}\label{hj}
\sum_{s=0}^a(-1)^sq^{s^2-s}
\left(
\begin{array}{c}
2b+1-2s\\2a-2s
\end{array}
\right)_q
\left(
\begin{array}{c}
b\\s
\end{array}
\right)_{q^2}
=
q^{2a^2-a}\frac{(1-q^{2b-2a+2})}{(1-q^{2b+2})}\frac{(q^{2b-4a+4};q^2)_{2a}}{(q;q)_{2a}}
\end{equation}
for $a=k-j$ and $b=\frac 12(n-1)-j$. 
The formula \eqref{hj} is proved by rewriting the left hand side as the basic hypergeometric function
along the lines of Remark \ref{tohyper} by observing that the ratio of consecutive terms  
can be written in the form
$$
\frac {c_{s+1}}{c_s} 
=q^{4a-2b-2}\frac{(1-q^{2s-2a})(1-q^{2s-2a+1})}{(1-q^{2s+2})(1-q^{2s-2b-1})}.
$$ 
Thus, we can simplify the left hand side of \eqref{hj} to
$$
\left(
\begin{array}{c}
2b+1\\2a
\end{array}
\right)_q
{~}_2\phi_1
\left(
\begin{array}{cc}
q^{-2a},q^{-2a+1}\\
q^{-2b-1}
\end{array};q^2,q^{4a-2b-2}
\right).
$$
We then employ  \cite[equation 1.5.2]{GR}, which is a particular case of Heine's analog 
of Gauss'  summation formula.
We thank Hjalmar Rosengren for pointing out this 
simplification.

\smallskip
Uniqueness follows from the fact that the coefficient by $f_{k,i,n}$ 
is $1$,  so the equations allow us to solve for $f_{1,i,n}$, then $f_{2,i,n}$, and so on.
\end{proof}

The most important result of this section is the formula for $f_{k,i,n}$ used in the proof of Theorem \ref{main.main}.
We thank Hjalmar Rosengren who observed that the identity we needed to prove follows from the Jain identity with
a particular choice of parameters.
\begin{proposition}\label{app.main}
There holds 
$$f_{k,i,n} =
\frac {q^{nk-1}-1}{q-1}\prod_{j=k+1}^{\frac 12(n-1)} \frac {q^{2j}-1}{q^{2j-2k}-1}
+q^{nk-1}\prod_{j=\frac 12(n-1)-k-i+1}^{\frac 12(n-1)-i} \frac{q^{2j}-1}{q^{2j-n+1+2k+2i}-1}.
$$
\end{proposition}

\begin{proof}
In view of Proposition \ref{newcor} it is enough to show that the above formula for $f_{k,i,n}$ fits into the 
the equation of the proposition. In view of Propositions \ref{g} and \ref{l}, this amounts to
a certain $q$-hypergeometric identity. We observe that the above formula for $f_{k,i,n}$ can be 
extended to $f_{0,i,n}$ to give $0$. Then the identity of Proposition \ref{newcor} is
$$
A_{k,n}+B_{k,i,n}=C_{k,n}+D_{k,i,n}
$$
where
\begin{align*}
&A_{k,n}=\sum_{j=0}^k 
\frac {q^{nj-1}-1}{q-1}
\left(\begin{array}{c}
\frac 12(n-1)\\ j
\end{array}
\right)_{q^2}
q^{2(k-j)^2-(k-j)}
\frac
{(1-q^{n+1-2k})}
{(1-q^{n+1-2j})}
\frac {(q^{n+3-4k+2j};q^2)_{2k-2j}}{(q;q)_{2k-2j}}
\\
&
B_{k,i,n}
=\sum_{j=0}^k
\left(
\begin{array}{c}
\frac 12(n-1)-i\\ j
\end{array}
\right)_{q^2}
q^{2(k-j)^2-(k-j)+nj-1}
\frac
{(1-q^{n+1-2k})}
{(1-q^{n+1-2j})}
\frac {(q^{n+3-4k+2j};q^2)_{2k-2j}}{(q;q)_{2k-2j}}\\
\\
&
C_{k,n}
= \frac {q^{2k^2-k-1}-1}{q-1}
\left(
\begin{array}{c}
 n\\ 2k
\end{array}
\right)_{q}
\\
&
D_{k,i,n}=q^{2k^2-k-1} 
\sum_{0\leq r\leq 2k}
\left(
\begin{array}{c}
n-2i\\
r
\end{array}
\right)_{q}
q^{(2k-r)(n-2i-r)}
\frac{ \prod_{j=i+r+1-2k}^{i} (1-q^{2j})}{\prod_{j=1}^{2k-r} (1-q^{j})}.
\end{align*}

We claim that 
$$A_{k,n}=C_{k,n},~~B_{k,i,n}=D_{k,i,n}.$$
To prove the first identity, we switch to the summation over $s=k-j$ in $A_{k,n}$ and put it into
$q$-hypergeometric form with base $q^2$ along the lines of Remark \ref{tohyper} to get
\begin{align*}
\begin{split}
&(1-q)A_{k,n}
=
{\left(
\begin{array}{c}
\frac 12(n-1)\\
k
\end{array}
\right)_{q^2}}
\left(
{~}_2\phi_1
\left(
\begin{array}{c}
q^{-2k},~~q^{-n-1+2k}\\
\hskip -30pt q
\end{array}
;q^2,q^{n+2}
\right)
-
q^{nk-1}\cdot
\right.
\\
&\left.
\cdot
{~}_2\phi_1
\left(
\begin{array}{c}
q^{-2k},~~q^{-n-1+2k}\\
\hskip -30pt q
\end{array}
;q^2,q^{2}
\right)
\right)
=(1-q^{2k^2-k-1})
{\left(
\begin{array}{c}
\frac 12(n-1)\\
k
\end{array}
\right)_{q^2}}
\frac{(q^{n+2-2k};q^2)_k}{(q;q^2)_k}
\end{split}
\end{align*}
by using Heine's formulas \cite[equations 1.5.2-1.5.3]{GR}. This then implies $A_{k,n}=C_{k,n}$ by a straightforward calculation which we
leave to the reader.

\smallskip
Since $B_{k,i,n}$ and $D_{k,i,n}$ are rational functions of $q^n$, it suffices to verify that they are equal for 
all sufficiently large $n$ for fixed $k$ and $i$. We have 
$$
B_{k,i,n} = q^{2k^2-k-1}
\frac{(1-q^{n+1-2k})(q^{n+3-4k};q^2)_{2k}}{(1-q^{n+1})(q;q)_{2k}}
{~}_3\phi_2
\left(
\begin{array}{c}
q^{-2k},q^{1-n+2i},q^{1-2k}\\
q^{1-n},  q^{n+3-4k}
\end{array}
;q^2,q^{n+2-2i}
\right).
$$
Again, the easiest way to prove this is by comparing the ratio of the terms for $j+1$ and $j$  to 
the corresponding ratio of the $q$-hypergeometric series in Remark \ref{tohyper}.
We then use the transformation formula for ${~}_3\phi_2$ \cite[Appendix, equation III.13]{GR}
with parameters $(b,c,d,e,n,q)\to (q^{-n+1+2i},q^{1-2k}, q^{1-n}, q^{n+3-4k},k, q^2)$
to  rewrite $B_{k,i,n}$ as
$$
B_{k,i,n} =  q^{2k^2-k-1}
\frac{(1-q^{n+1-2k})(q^{n+3-4k};q^2)_{2k}}{(1-q^{n+1})(q;q)_{2k}}\cdot
$$
$$
\cdot \frac{(q^{n+2-2k};q^2)_k}{(q^{n+3-4k};q^2)_k}
{~}_3\phi_2
\left(
\begin{array}{c}
q^{-2k},q^{1-2k},q^{-2i}\\
q^{1-n},q^{-n}
\end{array}
;q^2,q^2
\right).
$$

\smallskip
Similarly, for large enough $n$ we have  (by switching from $r$ to $(2k-r)$ in the summation)
$$
D_{k,i,n}=
q^{2k^2-k-1}
\left(
\begin{array}{c}
n-2i\\
2k
\end{array}
\right)_{q}
{~}_3\phi_1
\left(
\begin{array}{c}
q^{-2k},~~~q^{-i},~~-q^{-i}\\
 q^{n+1-2i-2k}
\end{array}
;q,-q^{n+1}
\right)
.
$$
We use Jain's equation
 \cite{Jain}, \cite[Exercise 3.4]{GR} 
with the parameters 
$
(a,b,d,n,q) \to (q^{-2i},q^{-i},q^{-n},2k,q)
$
to rewrite
$$
D_{k,i,n}=q^{2k^2-k-1}
\left(
\begin{array}{c}
n-2i\\
2k
\end{array}
\right)_{q}
\frac
{(q^{-n};q)_{2k}}
{(q^{-n+2i};q)_{2k}}
q^{4ki}
{~}_3\phi_2
\left(
\begin{array}{c}
q^{-2k},~~~q^{1-2k},~~q^{-2i}\\
 q^{-n},~~~q^{1-n}
\end{array}
;q^2,q^2
\right)
$$
$$
=q^{2k^2-k-1}
\left(
\begin{array}{c}
n\\
2k
\end{array}
\right)_{q}
{~}_3\phi_2
\left(
\begin{array}{c}
q^{-2k},~~~q^{1-2k},~~q^{-2i}\\
 q^{-n},~~~q^{1-n}
\end{array}
;q^2,q^2
\right)
.
$$
Thus, it remains to verify that 
$$
\frac{(1-q^{n+1-2k})(q^{n+3-4k};q^2)_{2k}}{(1-q^{n+1})(q;q)_{2k}}
 \frac{(q^{n+2-2k};q^2)_k}{(q^{n+3-4k};q^2)_k}
=
\left(
\begin{array}{c}
n\\
2k
\end{array}
\right)_{q}
$$
which is a straightforward calculation left to the reader.
\end{proof}


\begin{thebibliography}{99}

\bibitem{AKMW} D. Abramovich, K. Karu, K. Matsuki, J. W\l odarsczyk,
{\em Torification and Factorization of Birational Maps}, 
J. Amer. Math. Soc. (15)(2002), no. 3, 531--572.
%


\bibitem{Batyrev.1} V.V. Batyrev, {\em
Stringy Hodge numbers of varieties with Gorenstein canonical singularities},
Integrable systems and algebraic geometry (Kobe/Kyoto, 1997), 1--32, 
World Sci. Publishing, River Edge, NJ, 1998. 
%

\bibitem{BB}
V. Batyrev, L. Borisov, \emph{
Dual cones and mirror symmetry for generalized Calabi-Yau manifolds,} 
Mirror symmetry, II, 71-86, AMS/IP Stud. Adv. Math., 1, Amer. Math. Soc., Providence, RI, 1997.

\bibitem{Batetall}
V. Batyrev, I. Ciocan-Fontanine, B. Kim, D. van Straten,  \emph{Mirror Symmetry and Toric Degenerations of Partial Flag Manifolds,}
Acta Mathematica, 184(1) (2000), 1--39.

%\bibitem{Batyrev} V.V. Batyrev, 
%{\em Non-Archimedean integrals and stringy Euler numbers 
%of log-terminal pairs}, J. Eur. Math. Soc. (JEMS) {\bf 1} (1999), 
%no. 1, 5--33.
%

\bibitem{Bertram}
A. Bertram, \emph{An application of a log version of the Kodaira vanishing theorem to embedded projective varieties,} 
preprint alg-geom/9707001.


\bibitem{Bzd}
L. Borisov, \emph{Class of the affine line is a zero divisor in the Grothendieck ring,}
preprint  arXiv:1412.6194.

\bibitem{torloc}
L. Borisov, \emph{String cohomology of a toroidal singularity.} J. Algebraic Geom. 9 (2000), no. 2, 289--300.

\bibitem{BC}
L. Borisov, A. C\v ald\v araru,
\emph{The Pfaffian-Grassmannian derived equivalence.} J. Algebraic Geom. 18 (2009), no. 2, 201--222.

\bibitem{ellgen}
L. Borisov, A. Libgober, \emph{Elliptic genera of singular varieties.} Duke Math. J. 116 (2003), no. 2, 319--351.

\bibitem{BE} D. Buchsbaum, D. Eisenbud, {\em Algebra structures for finite free resolutions, 
and some structure theorems for ideals of codimension 3}, Amer. J. Math. 99 (1977), no. 3, 447--485. 
%

\bibitem{CKM}
H. Clemens, J. Koll\'ar, S. Mori, \emph{Higher-dimensional complex geometry. }
Ast\'erisque No. 166 (1988), 144 pp.

\bibitem{DR}
D. Dais, M. Roczen, \emph{On the string-theoretic Euler numbers of 3-dimensional A-D-E singularities.}
Adv. Geom. 1 (2001), no. 4, 373--426. 

\bibitem{DKh}
V. Danilov, A. Khovanskii, \emph{Newton polyhedra and an algorithm for computing Hodge-Deligne
numbers.}  Bull. AMS, Vol. 30 (1994), no 1 , 62--69.

\bibitem{GR}
G. Gasper, M. Rahman, \emph{Basic hypergeometric series.} With a foreword by Richard Askey. Second edition. 
Encyclopedia of Mathematics and its Applications, 96. Cambridge University Press, Cambridge, 2004. 

%
\bibitem{HoriTong}
K. Hori, D. Tong, \emph{Aspects of non-abelian gauge dynamics in two-dimensional N=(2,2) theories.} 
J. High Energy Phys. 2007, no. 5, 079, 41 pp.

%
\bibitem{Hartshorne}
R. Hartshorne, \emph{Algebraic geometry.} Graduate Texts in Mathematics, No. 52. Springer-Verlag, 
New York-Heidelberg, 1977. xvi+496 pp.

\bibitem{IM}
A. Iliev, L. Manivel, \emph{Fano manifolds of CalabiÐYau Hodge type.} J. Pure Appl. Algebra 219 (2015), no. 6, 2225--2244.

\bibitem{Jain}
V.K. Jain, \emph{Some transformations of basic hypergeometric functions. II.} SIAM J. Math. Anal. 12 (1981), no. 6, 957--961.

\bibitem{JPW}
T. Jozefiak, P. Pragacz, J. Weyman, \emph{Resolutions of determinantal
varieties and tensor complexes associated with symmetric and antisymmetric
matrices.} Young tableaux and Schur functors in algebra and geometry (Toru,
1980), pp. 109-189, Ast\'erisque, 87-88, Soc. Math. France, Paris, 1981.

\bibitem{Kuznetsov06}
A. Kuznetsov, \emph{Homological projective duality for Grassmannians of lines.}
preprint arXiv:math/0610957.

\bibitem{Kuznetsov}
A. Kuznetsov, \emph{Lefschetz decompositions and categorical resolutions of singularities.}
Selecta Math. (N.S.) 13 (2008), no. 4, 661--696.

\bibitem{Kuznetsov2}
A. Kuznetsov,
\emph{Semiorthogonal decompositions in algebraic geometry.}
preprint arXiv:1404.3143.

%
\bibitem{Rodland}
E. R\o dland, \emph{The Pfaffian Calabi-Yau, its mirror, and their link to the Grassmannian G(2,7).}
Compositio Math. 122 (2000), no. 2, 135--149. 
%
\bibitem{Thaddeus}
M. Thaddeus, \emph{Complete collineations revisited.}
Math. Ann. 315 (1999), no. 3, 469--495. 

\end{thebibliography}
\end{document}